\documentclass[12pt]{article}
\usepackage{amssymb,amsmath,amsthm,secdot,bbm,mathrsfs}

 \usepackage[normalem]{ulem}
\usepackage[bookmarks=true, bookmarksnumbered=true, colorlinks=true, pdfstartview=FitV, linkcolor=blue, citecolor=blue, urlcolor=blue]{hyperref}
\usepackage[nameinlink,capitalize]{cleveref}

\usepackage[shortlabels]{enumitem}

 \usepackage{crossreftools}
\newcommand*{\Gref}[1]{\hyperref[SPDEfunc]{Eq($#1$)}}

\DeclareMathOperator{\1}{\mathbbm{1}}
\DeclareMathOperator{\I}{\mathbbm{1}}

\DeclareMathOperator{\Leb}{Leb}

\newcommand{\be}{\begin{equation}}
\newcommand{\ee}{\end{equation}}
\def\E{\hskip.15ex\mathsf{E}\hskip.10ex}
\def\P{\mathsf{P}}
\def\Var{\mathop{\mbox{\rm Var}}}

\def\eps{\varepsilon}
\def\phi{\varphi}

\newcommand{\DN}{{\substack{x\in D\\|x|\le N}}}

\newtheorem{theorem}{Theorem}[section]
\newtheorem{lemma}[theorem]{Lemma}
\newtheorem{proposition}[theorem]{Proposition}
\newtheorem{corollary}[theorem]{Corollary}

\theoremstyle{definition}
\theoremstyle{definition}\newtheorem{remark}[theorem]{Remark}
\theoremstyle{definition}\newtheorem{definition}[theorem]{Definition}

\newtheorem{convention}[theorem]{Convention}

\crefname{Lemma}{Lemma}{Lemmas}
\crefname{Theorem}{Theorem}{Theorems}

\crefrangeformat{proposition}{Propositions~#3#1#4--#5#2#6}

\numberwithin{equation}{section}

\renewcommand{\ge}{\geqslant}
\renewcommand{\le}{\leqslant}

\newcommand{\nn}{\nonumber}

\newcommand{\wt}{\widetilde}

\renewcommand{\d}{\partial}

\newcommand{\0}{{D}}

\newcommand{\per}{{per}}
\newcommand{\neu}{{Neu}}

\newcommand{\A}{\mathcal{A}}

\newcommand{\C}{\mathcal{C}}

\newcommand{\F}{\mathcal{F}}

\renewcommand{\H}{\mathcal{H}}
\newcommand{\N}{\mathbb{N}}
\newcommand{\Q}{\mathbb{Q}}
\newcommand{\R}{\mathbb{R}}
\renewcommand{\S}{\mathcal {S}}
\newcommand{\V}{\mathcal{V}}
\newcommand{\Z}{\mathbb{Z}}
\def\K{{\mathcal K}}

\def\({\lt(}
\def\){\rt)}

\newcommand{\la}{\langle}
\newcommand{\ra}{\rangle}
\newcommand{\lt}{\left}
\newcommand{\rt}{\right}

\renewcommand{\L}[1]{{L_{#1}}}
\newcommand{\Lm}{{\L{m}(\Omega)}}

\newcommand{\Bes}{\mathcal{B}} 

\newcommand{\se}{{\S(D),m}}

\newcommand{\Bsp}{{\mathbf{B}}}

\newcommand{\Ctimespace}[3]{{\C^{#1}L_{#2}({#3})}}

\newcommand{\Ctimespacezerom}[1]{\Ctimespace{0}{m}{#1}}

\newcommand{\varz}{\rho}

\newcommand{\Di}[1]{|#1|}

\begin{document}

\makeatletter
\let\@fnsymbol\@arabic
\makeatother

\title{Analytically weak and mild solutions to stochastic heat equation with irregular drift}

\author{Siva Athreya%
	\thanks{International Centre for Theorertical Sciences -- TIFR, Survey No. 151, Shivakote, Hesaraghatta Hobli, Bengaluru - 560 089, India.
		Email: \texttt{athreya@icts.res.in}}
  \and
 Oleg Butkovsky%
 \thanks{Weierstrass Institute, Mohrenstrasse 39, 10117 Berlin, FRG. Email: \texttt{oleg.butkovskiy@gmail.com}}$^{\,\,\,,\!\!}$
 \thanks{Institut für Mathematik, Humboldt-Universität zu Berlin, Rudower Chaussee 25,
 	12489, Berlin,	FRG.}
 \and
 Khoa  L\^e
 \thanks{School of Mathematics, University of Leeds, U.K. Email: \texttt{k.le@leeds.ac.uk}}
  \and
Leonid Mytnik%
 \thanks{Technion -- Israel Institute of Technology,
Faculty of Data and Decision Sciences,
 Haifa, 3200003, Israel.  Email: \texttt{leonid@ie.technion.ac.il}
}
}

\maketitle

\begin{abstract}
Consider the stochastic heat equation
\begin{equation*}
\partial_t u_t(x)=\frac12 \partial^2_{xx}u_t(x) +b(u_t(x))+\dot{W}_{t}(x),\quad t\in(0,T],\, x\in D,
\end{equation*}
where $b$ is a generalized function, $D$ is either $[0,1]$ or $\R$, and $\dot W$ is space-time white noise on $\R_+\times D$. If the drift $b$ is a sufficiently regular function, then it is well-known that any analytically weak solution to this equation is also analytically mild, and vice versa. We extend this result to drifts that are generalized functions, with an appropriate adaptation of the notions of mild and weak solutions. As a corollary of our results, we show that for $b\in L_p(\R)$, $p\ge1$, this equation has a unique analytically weak and mild solution, thus extending the classical results of Gy\"ongy and Pardoux (1993). 
\end{abstract}

\section{Introduction}
We consider the stochastic heat equation: 
\begin{equation}\label{SPDE}
\partial_t u_t(x)=\frac12 \partial^2_{xx}u_t(x) +b(u_t(x))+\dot{W}_{t}(x),\quad t\in[0,T],\, x\in D,
\end{equation}
where $b$ is a generalized function (distribution) in the Besov space $\C^\alpha(\R,\R)$, $\alpha \in \R$, the domain $D$ is either $[0,1]$ or $\R$, $T > 0$, and $\dot{W}$ is space-time white noise on $[0,T] \times D$. When the drift $b$ is a measurable function, it is standard to define a solution to \eqref{SPDE} as either an analytically weak or an analytically mild solution, see \cref{Def:mildsol,Def:weaksol}. If $b$ is sufficiently regular, it is well-known that these two notions coincide. However, when $b$ is a generalized function, one cannot immediately make sense of these two notions, since it is unclear how to handle the term $b(u_t(x))$. To overcome this problem, \cite{ABLM24} introduced a notion of a solution, which we refer here as a mild regularized solution, see \cref{Def:mildrsol}. This notion coincides with the analytically mild solution if the drift is sufficiently regular. In this paper, we extend the notion of an analytically weak solution to the case where the drift $b$ is a generalized function and call it a weak regularized solution, see \cref{Def:weakrsol}. \cref{thm:main0}
shows that in the entire range where the existence of a mild regularized solution is known, it coincides with the weak regularized solution. As a corollary of this result, we extend seminal works \cite{GP93a,GP93b} and show that if $b\in L_p(\R)$, $p\ge1$, then equation \eqref{SPDE} has a unique analytically weak and analytically mild solution, see \cref{thm:main1}. Previously, the existence of a unique solution was known only for $p = \infty$ (when $b$ is bounded) for mild solutions and for $p \ge 2$ for analytically weak solutions.

Let us recall that classically in the literature (see e.g. \cite{W86}), the above equation has been well-studied for the case when $b\colon \R \to \R$ is a measurable function. 
Existence and uniqueness of strong solutions were proved in \cite{W86} for the equations on a bounded domain $D$ and for $b$ being a Lipschitz continuous function with no more than linear growth. Later, strong existence and uniqueness of an analytically weak solution were established in \cite{GP93a,GP93b} under much weaker assumptions: namely, when the drift $b$ is the sum of a bounded function and an $L_p$-integrable function, $p > 2$. Path--by--path uniqueness of solutions to \eqref{SPDE} for measurable bounded $b$ was obtained recently in \cite{BM19}.  It is  known that  analytically mild and weak notions of solutions are equivalent if the drift $b$ is locally bounded ant is of at most most polynomial growth \cite{bib:Iw87}, see also \cite{SH94}. A simple proof of this result for bounded drifts can be found in \cite[Proposition~3.2]{Par21}.

Stochastic heat equation \eqref{SPDE} with distribution-valued rather than function-valued drift has attracted a lot of attention in  recent years. 
In this case, equation ~\eqref{SPDE}  is not well-posed in the
 standard sense: indeed  if  $b$ is not a function but only a
 distribution then  the drift $b(u_t(x))$ is a priori not well-defined pointwise.
 An interesting case of $b$ being a measure has been considered in \cite{BZ14}.
 A natural notion of solution to this equation in the spirit
 of \cite[Definition~2.1]{BC} was introduced in \cite{ABLM24} via the approximation of drift with a sequence of smooth functions.  This procedure gives  rise to what we will call a {\it mild regularized} solution to  SPDE~\eqref{SPDE}.  It  was shown that equation~\eqref{SPDE} has a unique strong (in the probabilistic sense) {\it mild regularized} solution if $b\in\Bes^\alpha_q$,
 $\alpha-\frac1q\ge-1$, $\alpha>-1$ and $q\in[1,\infty]$  \cite[Theorem 2.6]{ABLM24}. Weak (in the probabilistic sense) existence and uniqueness of a {\it mild regularized} solution if $b\in\C^\alpha$, $\alpha>-\frac32$, was obtained in \cite[Theorem~2.13]{BM24}. Stochastic heat equations with distributional drift and multiplicative noise were studied in a very recent paper \cite{D24}. However, no analogues of analytically weak solutions for equations with distributional drift have been considered in the literature.

This leads to several natural questions: if $b$ is a distribution, how should we define a \textit{weak regularized} solution that would be analogous to the analytically weak solution? Is the weak regularized solution equivalent to the mild regularized solution? If $b$ is an integrable function, are regularized solutions equivalent to standard solutions?  Finally, are mild and weak solutions equivalent if $b$ is not a locally bounded function? These are the questions that we will address in this paper. More precisely, in  \cref{Def:weakrsol}, we introduce the notion of a \textit{weak regularized} solution. We show the equivalence of weak and mild regularized solutions in \cref{thm:main0} under the same conditions for which the existence of a solution is known. Finally, \cref{thm:main2} and  \cref{thm:main1} show that all four notions of solutions are equivalent if $b$ is an integrable function.


\textbf{Convention on constants}.  Throughout the paper $C$ denotes a positive constant whose value may change from line to line; its dependence is always specified in the corresponding statement. 

\textbf{Acknowledgements}. The work of SA was supported in part by KE grant from the Department of Atomic Energy, Government of India, under project no. RTI4001.
 OB is  funded by the Deutsche Forschungsgemeinschaft (DFG, German Research Foundation) under Germany's Excellence Strategy --- The Berlin Mathematics Research Center MATH+ (EXC-2046/1, project ID: 390685689, sub-project EF1-22) and DFG CRC/TRR 388 ``Rough
Analysis, Stochastic Dynamics and Related Fields", Project B08. LM is supported in part by ISF grant No. ISF 1985/22.

\section{Main results}\label{s:MR}

Let us introduce the main notation.  Let $\C_b^\infty=\C^\infty_b(\R,\R)$ be the space of infinitely
differentiable real functions on $\R$ which are bounded and have bounded derivatives of all orders. For $\beta\in\R$, $q\in[1,\infty]$, let $\Bes^\beta_q$ denote the (nonhomogeneous) Besov space $\Bes^\beta_{q,\infty}(\R)$ of regularity $\beta$ and integrability $q$, see, e.g., \cite[Section~2.3.1]{T}. If $q=\infty$, we write $\C^\beta=\Bes^\beta_{\infty,\infty}$. Let $\Bsp(\0)$ be the space of bounded measurable functions on $\0$.  Let $g_t$, $t>0$ be the heat kernel on $\R$, that is
$$
g_t(x):=\frac{1}{\sqrt{2\pi t}}e^{-\frac{x^2}{2t}},\quad t>0,\,x \in \R,
$$
and let $G$ be the corresponding semigroup.
Let  $p_t^{\per}$ and $p_t^{\neu}$ be the  the heat kernels on $[0,1]$ with the periodic and Neumann boundary conditions, respectively. That is,
\begin{align*}
	&p_t^\per(x,y):=\sum_{n\in\Z} g(t,x-y+n),\quad t>0,\,x,y\in [0,1];\\
	&p_t^\neu(x,y):=\sum_{n\in\Z} (g(t,x-y+2n)+g(t,x+y+2n)),\quad t>0,\,x,y\in [0,1].
\end{align*}

As in \cite{ABLM24}, we consider the stochastic heat equation \eqref{SPDE} in three setups: on the entire real line, on the interval $[0,1]$ with periodic boundary conditions, and on the interval $[0,1]$ with Neumann boundary conditions. Let $\C_{per}^\infty([0,1])$ denote the space of infinitely differentiable periodic functions on $[0,1]$. Let $\C_{Neu}^\infty([0,1])$ denote the space of infinitely differentiable functions on $[0,1]$ satisfying $f'(0) = f'(1) = 0$ for any $f \in \C^\infty_{Neu}([0,1])$. Finally, let $\S_{Sch}(\R)$ denote the Schwartz space of rapidly decreasing, infinitely differentiable functions.

To simplify the presentation of our results we introduce the following notational convention.

\begin{convention}\label{c:conv}
Further, the triple $(D, p, \S(D))$ stands for one of the following three options: $(\R, g, \S_{Sch}(\R))$, $([0,1], p^{\per}, \C^{\infty}_{per}([0,1]))$, or $([0,1], p^{\neu}, \C^{\infty}_{Neu}([0,1]))$. The corresponding semigroup is denoted by $P$.
\end{convention}

For measurable functions $f,g\colon D\to\R$ we write as usual
\begin{equation*}
\langle f,g\rangle:=\int_D f(x) g(x)\,dx.	
\end{equation*}

Fix now the length of the time interval $T>0$, a probability space $(\Omega,\F,\P)$ equipped with a complete filtration $(\F_t)$. We recall that a Gaussian process $W\colon L_2(D)\times[0,T]\times\Omega\to\R$ is called \textit{$(\F_t)$-space-time white noise} if for any $\phi,\psi\in L_2(D)$ the process $(W_t(\phi))_{t\in[0,T]}$ is an $(\F_t)$--Brownian motion and $\E W_s(\psi)W_t(\phi)=(s\wedge t)\int_D\phi(x)\psi(x)\,dx$, where $s,t\in[0,T]$.

We define now the stochastic convolution process (the solution to the linear stochastic heat equation)
\begin{equation}\label{def.V}
V_t(x):=\int_0^t\int_{\0}p_{t-r}(x,y)W(dr,dy),\quad t\ge0,\,x\in\0,
\end{equation}
where the integration in \eqref{def.V} is the Wiener integral, see, e.g., \cite[Section~1.2.4]{Marta}.
Recall the definition of the analytically mild solution.

\begin{definition}\label{Def:mildsol}
Let $b: \R \rightarrow \R$ be a measurable function and $u_0\in\Bsp(\0)$. A jointly continous adapted  process
$u\colon(0,T]\times\0\times\Omega\to\R$ is called an analytically \textit{mild solution of \eqref{SPDE} with the initial condition $u_0$} if for any $x\in\0$, $t \in [0,T]$ we have
\begin{align}
&\int_0^t\int_\0 p_{t-r}(x,y) |b|(u_r(y))dydr<\infty,\quad a.s.;\label{SPDEint}\\
 &u_t(x)=P_tu_0(x)+\int_0^t\int_\0 p_{t-r}(x,y) b(u_r(y))dydr + V_t(x),\quad a.s. \label{SPDEfunc}
\end{align}
\end{definition}

Sometimes this concept of solution is called a `random field' solution. If there is no risk of confusion, we will further sometimes refer to analytically mild solutions as just mild solutions.

\begin{remark}
When the function $b$ is unbounded it is not clear at all whether there exists a set of full probability  measure $\Omega'\subset \Omega$ such that the integral in \eqref{SPDEint} is finite 
on $\Omega'$  for all $x\in\0$, $t \in [0,T]$. Therefore, it is not obvious that there exists a universal set $\Omega'$, on which 
 identity \eqref{SPDEfunc} holds  for all $x\in\0$, $t \in [0,T]$.  Here we follow the classical definition of a mild solution, see, e.g., \cite[Definition~4.1.1]{Marta}, which allows the good set $\Omega'$ to depend on $t,x$. 
\end{remark}

Next, we consider a related notion of a solution for the case when $b$ is a generalized function in the Besov space $\Bes^\beta_{q}$, $\beta \in\R$, $q\in[1,\infty]$. This notion was introduced in \cite{ABLM24} motivated by the corresponding definition for SDEs \cite[Definition~2.1]{BC}. 

Let $f\in \Bes^\beta_q$, where $\beta\in\R$ and $q\in[1,\infty]$. We say that
a sequence of functions $(f_n)_{ n\in\Z_+}$, where $f_n\colon\R\to\R$,  converges to $f$ in
$\Bes^{\beta-}_q$ as $n\rightarrow \infty$ if
$\sup_{n\in\Z_+}\|f_n\|_{\Bes^\beta_q}<\infty$ and
\begin{equation*}
\lim_{n\rightarrow \infty} \| f_n-f\|_{\Bes^{\beta'}_q}=0,\quad \text{for any $\beta'<\beta$}.
\end{equation*}
It is clear that for any $f\in \Bes^\beta_q$, there is a sequence of
functions $(f_n)_{ n\in\Z_+} \subset \C^\infty_b$ such that
$f_n\rightarrow f$ in $\Bes^{\beta-}_q$ as $n\to\infty$, for example one can take $f_n:=G_{1/n} f$.

\begin{definition} \label{Def:mildrsol}
Let $\beta\in\R$, $q\in[1,\infty]$, $b\in\Bes_q^\beta$.   Let
$u_0\in\Bsp(\0)$. A jointly continuous adapted process
$u\colon(0,T]\times\0\times\Omega\to\R$ is called a \textit{mild regularized
    solution of \eqref{SPDE} with the initial condition $u_0$} if
  there exists a jointly continuous process $K\colon[0,T]\times\0\times\Omega\to\R$ such
  that
\begin{enumerate}
   \item[(1)] we have $\P$-a.s. 
   \begin{equation}\label{idmildid}
   u_t(x)=P_tu_0(x)+K_t(x)+V_t(x),\quad  x\in\0,\,\, t\in(0,T];
   \end{equation}  
           \item[(2)] for any sequence of functions $(b^n)_{n\in\Z_+}$ in
           $\C_b^\infty$ such that $b^n\to b$ in $\Bes_q^{\beta-}$ we
           have for any $N>0$ 
         \begin{equation}\label{idmildreg}
           \sup_{t\in[0,T]}\sup_\DN 
         \lt|\int_0^t\int_\0 p_{t-r}(x,y) b^n(u_r(y))\,dy\,dr-K_t(x)\rt|\to0\quad  \text{in probability as $n\to\infty$.}
         \end{equation}
\end{enumerate}
\end{definition} 
When $b\in\C^\beta$ with $\beta>0$, we can choose a sequence $(b_n)$
which converges to $b$  uniformly.  Then it is immediate that, in this case,
\cref{Def:mildrsol} is equivalent to the usual notion of a mild solution of
\eqref{SPDE} as in \cref{Def:mildsol}.

Our third definition of a solution is the classical weak solution. For $f\in L_2([0,T]\times D)$,  we put
\begin{equation*}
W_t(f) = \int_0^t\int_0^1 f(r,y) W(dr,dy),\quad t\in[0,T].
\end{equation*}


\begin{definition}\label{Def:weaksol}
Let $b: \R \rightarrow \R$ be a measurable function and $u_0\in\Bsp(\0)$. A jointly continuous adapted process $u\colon(0,T]\times\0\times\Omega\to\R$ is called an analytically {\em weak solution of } \eqref{SPDE} if for any $\phi \in \S(\0)$ we have $\P$-a.s. 
\begin{equation}\label{idweak}
\langle u_t,\phi\rangle= \langle u_0,\phi\rangle +\int_0^t \langle u_s, \frac{1}{2}\Delta\phi\rangle\,ds  +\int_0^t \langle b(u_s), \phi\rangle\,ds  +W_t(\phi),\qquad t\in[0,T].
\end{equation}
and all the integrals are assumed to be well-defined.
\end{definition}

We now introduce a new solution concept that extends the notion of  a weak solution of \eqref{SPDE} to the case when $b$ is a generalized function.

\begin{definition} \label{Def:weakrsol}
Let $\beta\in\R$, $q\in[1,\infty]$, $b\in\Bes_q^\beta$.   Let
$u_0\in\Bsp(\0)$. A jointly continuous adapted process
$u\colon(0,T]\times\0\times\Omega\to\R$ is called a \textit{ weak regularized solution  of  \eqref{SPDE}  with the initial condition $u_0$}
  if there exists a continuous in time process $\K\colon[0,T]\times\S(D)\times \Omega\to\R$  such that
\begin{enumerate}[(1)]
   \item for any $\phi\in\S(D)$ we have $\P$-a.s. 
   \begin{equation}\label{idweakreg}
    \langle u_t,\phi\rangle= \langle u_0,\phi\rangle +\int_0^t \langle u_s, \frac{1}{2}\Delta\phi\rangle\,ds  +\K_t(\phi)+W_t(\phi),\quad  t\in[0,T]; 
    \end{equation}
   \item for any sequence of functions $(b^n)_{n\in\Z_+}$ in $\C_b^\infty$ such that $b^n\to b$ in $\Bes_q^{\beta-}$
 we have for any  $\phi\in\S(D)$
\begin{equation}\label{Kcurlydef}
\sup_{t\in[0,T]}\lt|\int_0^t\int_\0 \phi(x) b^n(u_r(x))\,dx\,dr-\K_t(\phi)\rt|\to0\quad  \text{in probability as $n\to\infty$.}
\end{equation}
  \end{enumerate}
\end{definition}
\begin{remark}
If $\beta >0$  we can choose a sequence $(b^n)$ which converges to $b$  uniformly and the definition agrees with the usual notion of a weak solution as in Definition \ref{Def:weaksol}.  
\end{remark}

As it is standard in the analysis of SPDEs with distributional drift,  we restrict ourselves to solutions having certain regularity. We consider the following  class of processes initially introduced in \cite{ABLM24}. 
\begin{definition}\label{D:Def15}
Let $\kappa\in[0, 1]$. We say that a measurable adapted process $u\colon[0,T]\times\0\times\Omega\to\R$ belongs to the \textit{class $\V(\kappa)$} if for any $m\ge2$, $\sup_{(t,x)\in(0,T]\times\0}\|u_t(x)\|_{L_m(\Omega)}<\infty$ and 
\begin{equation}\label{kdefineq}
\sup_{0< s\le t\le T}\sup_{x\in\0}\frac{\|u_t(x)-V_t(x)-(P_{t-s}[u_s-V_s](x))\|_\Lm}{|t-s|^\kappa}<\infty.
\end{equation}
\end{definition}
\begin{remark}
If $u$ is a mild regularized solution to \eqref{SPDE}, then it can be decomposed as 
$$
u_t=P_tu_0+K_t+V_t,\quad t>0,
$$
 see \cref{Def:mildrsol}. In this case, the numerator in \eqref{kdefineq} is just  $\|K_t(x)-P_{t-s}K_s(x)\|_\Lm$. Thus, class $\V(\kappa)$ contains solutions of \eqref{SPDE} such that the moments of their drift satisfy certain regularity conditions.
\end{remark}

We are now ready to state our main results. The first result states  the equivalence of the notions of weak and mild regularized solutions.

\begin{theorem} 
 \label{thm:main0} Let $\alpha>-\frac{3}{2}$,  $b \in \C^\alpha$, $u_0\in\Bsp(D)$. Let  $u\colon[0,T]\times D\times\Omega\to\R$ be a measurable adapted process. Suppose that $u\in \V(5/8)$.  Then the following are equivalent:
 \begin{enumerate}[(a)]
  \item $u$ is a  weak regularized solution  to equation  \eqref{SPDE}. 
  \item  $u$ is  a mild regularized  solution to equation  \eqref{SPDE}.
 \end{enumerate}
\end{theorem}

We note that the assumption $\alpha>-\frac32$ and $u\in \V(5/8)$ are precisely the assumptions of  \cite[Theorem~2.6]{ABLM24}, which guarantees existense of a mild regularized solution to SPDE~\eqref{SPDE}.

We recall that the uniqueness theorem for equation \eqref{SPDE} \cite[Theorem~2.6]{ABLM24} is applicable to mild regularized solutions in the class $\V(3/4)$. Thus, in order to establish the existence and uniqueness of analytically mild solutions when $b \in L_p(\mathbb{R})$, $p \ge 1$, one must show that any analytically mild solution is a mild regularized solution and belongs to $\V(3/4)$. Similar problems appear in the proof of existence and uniqueness of analytically weak solutions to \eqref{SPDE}. These questions are addressed in the following theorem, which also shows the equivalence of all four notions of solutions when $b$ is an integrable function.

\begin{theorem}  \label{thm:main2} 
 Let $p\ge1$, $b\in L_p(\R)$, $u_0\in\Bsp(D)$.  Let  $u\colon[0,T]\times D\times\Omega\to\R$ be a measurable adapted process.
Then the following are equivalent:  

\begin{enumerate}[(a)]
\item $u$ is  an analytically  weak solution to  \eqref{SPDE}.
\item $u$ is a  weak regularized solution to \eqref{SPDE} in the class $\V(5/8)$.
\item $u$ is an analytically   mild solution to \eqref{SPDE}.
\item $u$ is  a mild regularized solution to \eqref{SPDE} in the class $\V(5/8)$.
\end{enumerate}
Moreover if any of (a)-(d) holds then $u\in \V(1-\frac1{4p})$.
\end{theorem}

We emphasize that in parts (a) and (c) of the theorem, we do not assume that $u$ is in the class $\V(\kappa)$. In fact, the non-trivial part of the proof is to show that any weak or mild solution to \eqref{SPDE} automatically belongs to $\V(1-\frac1{4p})\subset \V(3/4)$. This is achieved using the stochastic sewing lemma with random controls \cite[Theorem~4.7]{ABLM24}.

\begin{corollary} 
 \label{thm:main1} Let $p\ge1$, $b\in L_p(\R)$, $u_0\in\Bsp(D)$. 
 Then equation  \eqref{SPDE} has a unique analytically  mild solution and  
a unique  analytically weak solution which coincide with each other. The solutions are strong in the probabilistic sense.
\end{corollary}

\begin{remark}
We recall that the statements of \cref{thm:main0,thm:main2,thm:main1} are valid in three different settings: on a bounded domain with periodic boundary conditions, on a bounded domain with Neumann boundary conditions, and on $\R$, see \cref{c:conv}.
\end{remark}

Note that for $b \in L_p(\mathbb{R})$, the existence and uniqueness of analytically mild solutions were previously known only for $p = \infty$ \cite{GP93a}, and the existence and uniqueness of analytically weak solutions were established only for $p \ge 2$ \cite{GP93b}. \cref{thm:main1} improves these results, and shows the existence and uniqueness of both analytically mild and weak solutions for $p \ge 1$.

It is easy to show that if $u$ is an analytically weak solution to \eqref{SPDE}, then $u$ satisfies \cref{Def:mildsol} for Lebesgue a.e. $x\in D$. In particular, one can show that for any $t\in[0,T]$ for Lebesgue a.e. $x\in D$ one has 
\begin{equation*}
\int_0^t\int_\0 p_{t-r}(x,y) |b|(u_r(y))dydr<\infty,\quad a.s.
\end{equation*}
The challenging part is to show that for all $x\in D$ this integral is actually $\P$-a.s. finite.

\begin{remark}
	For the case $D = \R$, additional assumptions are sometimes imposed on the class of solutions. For example, \cite[Theorem~2.1]{SH94} requires that for any $t \geq 0$, we have $u_t \in \C_{tem}(\R)$ (the space of functions that grow slower than $\exp(\lambda |x|)$ for any $\lambda > 0$). A careful application of Mitoma's theorem \cite{mitoma} allows us to remove this restriction, see the proof of \cref{l:phitd}. 
\end{remark}



The rest of the paper is organized as follows. We place a number of auxiliary technical lemmas in \cref{s:3}. The proofs of the main results are given in \cref{s:4}.

\section{Sewing lemmas and integral bounds}\label{s:3}

First, we set up some necessary notation.
For $0\le S<T$ we denote by $\Delta_{[S,T]}$ the simplex
$$
\Delta_{[S,T]}:=\{(s,t)\in[0,T]^2: S\le s\le t \le T\}.
$$
Let $(\Omega,\F, (\F_t)_{t\ge0}, \P)$ be a complete filtered probability space on which the 
the processes below are defined. We will write $\E^s$ for the conditional expectation given $\F_s$
$$
\E^s[\cdot]:=\E[\cdot|\F_s],\quad s\ge0.
$$

Let $0\le S\le T$. Let $f\colon[S,T]\times\0\times \Omega\to\R$ be a measurable function. For $\tau\in(0,1]$, $m\ge1$ define
\begin{align*}
		&\|f\|_{\Ctimespacezerom{[S,T]}}:=\sup_{t\in[S,T]}\sup_{x\in D}\|f(t,x)\|_{L_m(\Omega)};\\
		&[f]_{\Ctimespace{\tau}{m}{[S,T]}}:=\sup_{(s,t)\in\Delta_{[S,T]}}\sup_{x\in D} \frac{\|f_t(x)-P_{t-s}f_s(x)\|_{L_m(\Omega)}}{|t-s|^\tau}.
\end{align*}

Recall the statement of stochastic sewing lemma with random controls.

\begin{definition}[{\cite[Definition~4.6]{ABLM24}}]\label{def.control}
Let $T>0$, $\lambda$ be a measurable function $\Delta_{[0,T]}\times \Omega\to \R_+$. We say that  $\lambda$ is a \textit{random control} if for any $0\le s\le u \le t\le T$ one has
	$$
	\lambda(s,u,\omega)+\lambda(u,t,\omega)\le \lambda(s,t,\omega)\quad \text{a.s.}
	$$
\end{definition}

\begin{proposition}[{\cite[Theorem~4.7]{ABLM24}}]\label{lem:B.rcontrol}
Let $m\in[2,\infty)$. Let $\lambda$ be a random control. Assume that there exist constants $\Gamma_1,\Gamma_2, \alpha_1,\alpha_2, \beta_2, \ge0$,  such that $\alpha_1>\frac12$, $\alpha_2+\beta_2>1$ and
\begin{align}
&\|\delta A_{s,u,t}\|_{L_m(\Omega)}\le \Gamma_1|t-s|^{\alpha_1};\label{con:wei.dA2}\\
&|\E^u \delta A_{s,u,t}|\le \Gamma_2|t-s|^{\alpha_2}\lambda(s,t)^{\beta_2}\quad\text{a.s.}\label{Eudeltaastbound}
\end{align}
for every $(s,t)\in\Delta_{[0,T]}$ and $u:=(s+t)/2$.
Further, suppose that there exists a process $\A=\{\A_t:t\in[0,T]\}$ such that for any $t\in[0,T]$ and any sequence of partitions $\Pi_N:=\{0=t^N_0,...,t^N_{k(N)}=t\}$ of $[0,t]$ such that $\lim_{N\to\infty}\Di{\Pi_N}\to0$  one has
	\begin{equation}\label{limpart}
	\A_t=\lim_{N\to\infty} \sum_{i=0}^{k(N)-1} A_{t^N_i,t^N_{i+1}}\,\, \text{in probability.}
	\end{equation}

Then there exists a measurable function  $B\colon\Delta_{[0,T]}\times\Omega\to \R_+$ and a constant $C=C(\alpha_1,\alpha_2,\beta_2, m)>0$, such that for every $(s,t)\in\Delta_{[0,T]}$
\begin{align*}
|\A_t-\A_s-A_{s,t}|&\le C\Gamma_2 |t-s|^{\alpha_2}\lambda(s,t)^{\beta_2}+ B_{s,t}\quad\text{a.s.};\\
\|B_{s,t}\|_{L_m}&\le C\Gamma_1 |t-s|^{\alpha_1}.
\end{align*}
\end{proposition}

We use the stochastic sewing lemma with random controls in order to bound integrals of a perturbation of the stochastic convolution $V$ under very loose assumptions on the perturbation. In particular, we do not assume that this perturbation has finite moments or any regularity.

\begin{lemma}\label{L:metaml}
Let $\gamma\in(-2,0)$, $T\in(0,1]$. Let $f:\R\to\R$ be a  measurable function. Suppose further that $f$ is bounded   or $f\ge0$. Let  $\psi\colon[0,T]\times\0\times\Omega\to\R$ be a measurable function adapted to the filtration $\{\F_t\}$. 
Assume that there exists a nonnegative stochastic process $w\colon\Delta_{[0,T]}\times\0\times\Omega\to[0,\infty)$ with the following properties: for any $t\in[0,T]$ there exists a set $D_t\subset D$ with  $\Leb(D\setminus D_t)=0$ such that  for any $(s,u)\in\Delta_{[0,t]}$,  $x\in D_t$
\begin{align}\label{wdef1}
&|\psi_t(x)-P_{t-s}\psi_s(x)|\le w_{s,t}(x)\quad  \text{$\P$-a.s.};\\
& P_{t-u}w_{s,u}(x)+w_{u,t}(x)\le w_{s,t}(x)\quad \text{$\P$-a.s.} \label{wdef2}
\end{align}
  Then there exists a nonnegative stochastic process $L\colon\Delta_{[0,T]}\times D\times\Omega\to[0,\infty)$ such that  
  for any $(s,t)\in\Delta_{[0,T]}$, $x\in D_T$
  \begin{equation}\label{Boundrc}
  \Bigl|\int_s^t\int_\0 p_{T-r}(x,y) f(V_r(y)+\psi_r(y))\,dydr\Bigr|\le 
  C\|f\|_{\C^\gamma}(t-s)^{\frac{3+\gamma}4}P_{T-t}w_{s,t}(x)+L_{s,t}(x),
  \end{equation}
  where $C=C(\gamma)$. 
  Further,  for any  $m\in[2,\infty)$, there exists a constant $C=C(\gamma,m)>0$ such that for any $x\in D_T$
\begin{equation}\label{Boundlst}
\|L_{s,t}(x)\|_{\Lm}\le C\|f\|_{\C^\gamma}|t-s|^{1+\frac\gamma4}.
\end{equation}
\end{lemma}

\begin{proof}

\textit{Step 1.} Let $f\in\C^1(\R)$. Fix $x\in D_T$, $T\in(0,1]$. 
We would like to apply stochastic sewing lemma with random controls (\cref{lem:B.rcontrol}). For $(s,t)\in\Delta_{[0,T]}$ set
\begin{align*}
&A_{s,t}(x):=\int_s^t\int_\0 p_{T-r}(x,y) f(V_r(y)+P_{r-s}\psi_s(y))\,dydr;\\
&\A_{t}(x):=\int_0^t\int_\0 p_{T-r}(x,y) f(V_r(y)+\psi_r(y))\,dydr.
\end{align*}
Let us verify that all the conditions of \cref{lem:B.rcontrol}
are satisfied for the random control
$$
\lambda_{s,t}(x):=P_{T-t}w_{s,t}(x),\quad (s,t)\in\Delta_{[0,T]}.
$$
Note that $\lambda$ is indeed a random control since for $0\le s\le u \le t\le T$  we have
\begin{equation}\label{eqae}
\lambda_{s,u}(x)+\lambda_{u,t}(x)=P_{T-u}w_{s,u}(x)+P_{T-t}w_{u,t}(x)
=P_{T-t}[P_{t-u}w_{s,u}+w_{u,t}](x).
\end{equation}
If $t<T$, then by \eqref{wdef2}, $P_{t-u}w_{s,u}(y)+w_{u,t}(y)\le w_{s,t}(y)$ for Lebesgue a.e. $y\in D$. Therefore,  $P_{T-t}[P_{t-u}w_{s,u}+w_{u,t}](x)\le P_{T-t}w_{s,t}(x)$. If $t=T$, then since $x\in D_T$ we have $P_{T-u}w_{s,u}(x)+w_{u,T}(x)\le w_{s,T}(x)$ almost surely. Thus, in both cases 
\eqref{eqae} implies that for $0\le s\le u \le t\le T$
\begin{equation*}
\lambda_{s,u}(x)+\lambda_{u,t}(x)\le P_{T-t}w_{s,t}(x)=\lambda_{s,t}(x),\quad \P-a.s.,
\end{equation*}
and therefore $\lambda$ is a random control, see \cref{def.control}.  

Let $(s,t)\in\Delta_{[0,T]}$, let $u:=(s+t)/2$. Then we have
$$\delta A_{s,u,t}(x)=\int_u^t\int_\0 p_{T-r}(x,y)[f(V_r(y)+P_{r-s}\psi_s(y))-
f(V_r(y)+P_{r-u}\psi_u(y))] \,dydr.
$$
Applying consequently the Fubini theorem and \cite[(C.7), Lemma~A.3(iv) and Lemma~C.3]{ABLM24}, we deduce that
\begin{align*}
  |\E^u \delta A_{s,u, t}(x)|&=\Bigl|\int_u^t\int_\0 p_{T-r}(x,y) \E^u[f(V_r(y)+P_{r-s}\psi_s(y))-
  f(V_r(y)+P_{r-u}\psi_u(y))]\,dydr\Bigr|\\
  &\le\int_u^t\int_\0 p_{T-r}(x,y)\|G_{\varz_{r-u}(x)}f\|_{\C^1}|P_{r-s}\psi_s(y)- P_{r-u}\psi_u(y)|\,dydr\\
  &\le C\|f\|_{\C^\gamma} \int_u^t\int_\0 p_{T-r}(x,y) |P_{r-s}\psi_s(y)- P_{r-u}\psi_u(y)| (r-u)^{(\gamma-1)/4}\,dydr,
\end{align*}
for $C=C(\gamma)$,
where we used the notation $\varz_t(x):=\Var(V_t(x))$ and recall that $G$ denotes the Gaussian semigroup. Since $r-u>0$, assumption \eqref{wdef1} implies for any $y\in D$
\begin{equation*}
|P_{r-u}\psi_u(y)-P_{r-s}\psi_s(y)|\le P_{r-u}|\psi_u(\cdot)-P_{u-s}\psi_s(\cdot)|(y)\le
P_{r-u}w_{s,u}(y).
\end{equation*}
Therefore, we get
\begin{align*}
|\E^u \delta A_{s,u,t}(x)|&\le C\|f\|_{\C^\gamma} \int_u^t\int_\0 p_{T-u}(x,y) w_{s,u}(y) (r-u)^{(\gamma-1)/4}\,dydr
\\&\le C\|f\|_{\C^\gamma}\lambda_{s,u}(x)(t-u)^{(\frac34+\frac\gamma4)},
\end{align*}
for $C=C(\gamma)$. 
As we have shown previously, $(s,t)\mapsto \lambda_{s,t}(x)$ is a random control. Hence, the above estimate verifies \eqref{Eudeltaastbound}.

For each integer $j\ge1$, define $\psi^{j,+}_t(x):=(\psi_t(x)\wedge j)\vee0$, $\psi^{j,-}_t(x):=((-\psi_t(x))\wedge j)\vee0$, $\psi^{j}_t(x):=
\psi^{j,+}_t(x)-\psi^{j,-}_t(x)$, $t\in[0,T]$, $x\in\0$. Clearly, $\|\psi^{j}\|_{\C^0 L_m([0,T])}<\infty$. By \cite[Lemma~5.1]{ABLM24}, we have for any $(s,t)\in\Delta_{[0,T]}$, $x\in\0$
\begin{equation*}
  \Bigl\|\int_s^t\int_\0 p_{T-r}(x,y) f(V_r(y)+P_{r-s}\psi^j_s(y))\,dydr\Bigr\|_{L_m(\Omega)}\le C\|f\|_{\C^\gamma}|t-s|^{1+\frac\gamma4}
\end{equation*}
for $C=C(\gamma)>0$. By the Lebesgue monotone convergence theorem, we have 
$$
\lim_{j\to\infty}P_{r-s}\psi^{j,+}_s(y)=P_{r-s}\psi^+_s(y),\quad 
\lim_{j\to\infty}P_{r-s}\psi^{j,-}_s(y)=P_{r-s}\psi^-_s(y)
$$
 for 
every $r,s,y$. Then by the Lebesgue dominated convergence theorem, we see that 
\begin{equation}\label{tmp.1107}
  \|A_{s,t}(x)\|_{L_m}=\Bigl\|\int_s^t\int_\0 p_{T-r}(x,y) f(V_r(y)+P_{r-s}\psi_s(y))\,dydr\Bigr\|_{L_m}\le C\|f\|_{\C^\gamma}|t-s|^{1+\frac\gamma4}.
\end{equation}
This estimate verifies assumption \eqref{con:wei.dA2}.

Finally, it remains to verify condition \eqref{limpart}. Fix $t\in[0,T[$ and let $\Pi:=\{0=t_0,t_1,...,t_k=t\}$ be an arbitrary partition of $[0,t]$. Denote by $\Di{\Pi}$ its mesh size. Then we have
\begin{align*}
  &\Bigl|\A_t-\sum_{i=0}^{k-1} A_{t_i,t_{i+1}}\Bigr|\\
  &\quad\le \sum_{i=0}^{k-1} \int_{t_i}^{t_{i+1}} \!\!\!\int_\0 p_{T-r}(x,y) \bigl|f(V_r(y)+\psi_r(y))- f(V_r(y)+P_{r-{t_i}}\psi_{t_i}(y))\bigr|\,dy dr\\
  &\quad\le \|f\|_{\C^1}\sum_{i=0}^{k-1} 
  \int_{t_i}^{t_{i+1}} \!\!\!\int_\0 p_{T-r}(x,y) |\psi_r(y)-P_{r-{t_i}}\psi_{t_i}(y)|\,dy dr\\
  &\quad\le \|f\|_{\C^1}\sum_{i=0}^{k-1} 
  \int_{t_i}^{t_{i+1}} \!\!\! \lambda_{t_i,r}(x)  dr\\
  &\quad\le \|f\|_{\C^1}\sum_{i=0}^{k-1} 
  \int_{t_i}^{t_{i+1}} \!\!\! \lambda_{t_i,t_{i+1}}(x)  dr\\
  &\quad\le \|f\|_{\C^1}\lambda_{0,t}(x)|\Pi|,
\end{align*}
where in the third inequality we used \eqref{wdef1} and that $T-r>0$, and in the last two inequalities we have used the fact that $\lambda$ is a control. Thus,  condition \eqref{limpart} holds.

Thus all the conditions of \cref{lem:B.rcontrol}
 are satisfied. Hence we have for any $(s,t)\in\Delta_{[0,T]}$, $x\in D_T$
\begin{align*}
&\Bigl|\int_s^t\int_\0 p_{T-r}(x,y) f(V_r(y)+\psi_r(y))\,dydr\Bigr|\\ 
&\quad=|\A_t(x)-\A_s(x)|\\
&\quad\le 
C\|f\|_{\C^\gamma}\lambda_{s,t}(x)(t-s)^{\frac34+\frac\gamma4}+B_{s,t}(x)+|A_{s,t}(x)|
\end{align*}
where $\|B_{s,t}(x)\|_{\Lm}\le C\|f\|_{\C^\gamma}|t-s|^{1+\frac\gamma4}$  uniformly in $x\in D_T$ for $C=C(\gamma,m)$. 
To obtain \eqref{Boundrc}, it suffices to put  $L_{s,t}(x)=B_{s,t}(x)+|A_{s,t}(x)|$ and use \eqref{tmp.1107}.

The extension from $f\in \C^1$ to a general $f$ goes along the same lines as in, e.g., \cite[proof~of~Lemma 4.6]{BLM23}.  
  
\textit{Step 2}.  Assume now that $f\colon\R\to\R$ is a bounded continuous function. For $k\in\N$ put $f_k:=G_{1/k}f$. Then it is clear that $f_k$ converges to $f$ pointwise and $ \|f_k\|_{\C^\gamma}\le \|f\|_{\C^{\gamma}}$. Since $f_k\in\C^1$ for any $k\in\N$, inequality \eqref{Boundrc} for such $f$  follows now directly from Step~1 and Fatou's lemma.

\textit{Step~3}.  Let $f=\1_U$ where $U\subset \R$ is an open set of Lebesgue measure~$\eps>0$. By Urysohn's lemma, there exists a sequence of bounded continuous functions $f_n\colon\R\to[0,1]$, $n\in\Z_+$ such that $0\le f_n\le \1_U$ for any $n\in\Z_+$ and $f_n(x)\to \1_U(x)$ for any $x\in\R$ as $n\to\infty$. Then $f_n(V_r(y)+\psi_r(y))\to \1_U(V_r(y)+\psi_r(y))$ as $n\to\infty$ for any $r\in[0,T]$, $y\in\0$, $\omega\in\Omega$.
We apply the results of Step~1 with $\gamma=-1$ and deduce for any  $s,t\in\Delta_{[0,T]}$, $x\in D_T$, there exists a random variable $L^n_{s,t}(x)$ such that
  \begin{equation}\label{updstep22}
	\Bigl|\int_s^t\int_\0 p_{T-r}(x,y) f_n(V_r(y)+\psi_r(y))\,dydr\Bigr|\le 
	C\|f_n\|_{\C^{-1}}w_{0,T}(x)+L^n_{s,t}(x),
\end{equation}
and for any $m\ge2$ we have $\|L^n_{s,t}(x)\|_{L_m(\Omega)}\le C\|f_n\|_{\C^{-1}}$ for $C=C(\gamma,m)$.

By the embedding $L_1(\R)\subset \C^{-1}$, we have for any $n\in\Z_+$, $\|f_n\|_{\C^{-1}}\le \|f_n\|_{L_1(\R)}\le \eps$. Thus, defining $L_{s,t}(x):=\lim\inf_{n\to\infty}L^n_{s,t}(x)$, we get by Fatou's lemma from \eqref{updstep22}
  \begin{equation*}
	\Bigl|\int_s^t\int_\0 p_{T-r}(x,y) \1_U (V_r(y)+\psi_r(y))\,dydr\Bigr|\le 
	C\eps w_{0,T}(x)+L_{s,t}(x),
\end{equation*}
and for any $m\ge2$ we have $\|L_{s,t}(x)\|_{L_m(\Omega)}\le C\eps$ for $C=C(m)$.

\textit{Step 4}. Let $f$ be now a bounded measurable function. Fix  arbitrary $\delta>0$ and put $M:=\|f\|_{L_\infty(\R)}$. 
By Lusin’s theorem, there exists a bounded  continuous function $f_{\delta}\colon\R\to\R$ such that $\Leb(\{f_{\delta}\neq f\})\le \delta$ and 
$\|f_{\delta}\|_{L_\infty(\R)}\le  M$. Let $U_\delta$ be an open set of measure $2\delta$ containing the set $\{ f_{\delta}\neq f\}$; such set exists since the Lebesgue measure is regular. Then for any $x\in\R$ we have 
\begin{equation*}
	|f_{\delta}(x)-f(x)|\le 2M\I_{U_\delta}(x).
\end{equation*}
and 
\begin{equation}\label{urdiff}
	\|f_{\delta}\|_{\C^\gamma}\le  \|f\|_{\C^\gamma}+\|f_{\delta} - f\|_{\C^\gamma}\le  \|f\|_{\C^\gamma}+\|f_{\delta} - f\|_{L_{|\gamma|^{-1}\vee1}(\R)}\\
	\le  \|f\|_{\C^\gamma}+2M \delta^{|\gamma|\wedge1}.
\end{equation}
Therefore, applying Step~2 of the proof to the continuous function $f_{\delta}$ and using the results of Step~3 we get that for any $(s,t)\in\Delta_{[0,T]}$, $x\in D_T$ there exists a random variable $L_{s,t}^{\delta}(x)$ such that
\begin{align*}
&\Bigl|\int_s^t\!\int_\0 p_{T-r}(x,y) f(V_r(y)+\psi_r(y))\,dydr\Bigr|\\
&\,\, \le \Bigl|\int_s^t\!\int_\0 p_{T-r}(x,y) f_{\delta}(V_r(y)+\psi_r(y))\,dydr\Bigr|+2M\int_s^t\!\int_\0 p_{T-r}(x,y) \1_{U_\delta}(V_r(y)+\psi_r(y))\,dydr\\
&\,\,\le 		C\|f_{\delta}\|_{\C^\gamma}(t-s)^{\frac{3+\gamma}4}P_{T-t}w_{s,t}(x)+L_{s,t}^{\delta}(x)+ CM\delta w_{0,T}(x)
\end{align*}
and $\|L_{s,t}^{\delta}(x)\|_{L_m(\Omega)}\le C\|f_{\delta}\|_{\C^\gamma}|t-s|^{1+\frac\gamma4}+C\delta$ for $C=C(\gamma,m)$. Define now $L_{s,t}(x):=\lim\inf_{\delta\to0}L^\delta_{s,t}(x)$.
By passing to the limit in the above bound as $\delta\to0$ and using \eqref{urdiff} and Fatou's lemma, we get \eqref{Boundrc} and \eqref{Boundlst}.

\textit{Step 5}. Suppose that  $f$ is a nonnegative measurable function $\R\to\R$. It is clear that a sequence of bounded functions $f\wedge N$, $N\in\N$, converges to $f$ pointwise. Note also that $\|f\wedge N\|_{\C^{\gamma}}\le \|f\|_{\C^{\gamma}}$ for any $N\in\N$. Therefore,  
applying  the results of Step~4 to a bounded function $f\wedge N$ we derive
that for any $(s,t)\in\Delta_{[0,T]}$, $x\in D_T$ there exists a random variable $L_{s,t}^N(x)$   such that  
\begin{equation}\label{bigNbound}
	\int_s^t\int_\0 p_{T-r}(x,y) (f(V_r(y)+\psi_r(y))\wedge N)\,dydr\le 
	C\|f\|_{\C^\gamma}(t-s)^{\frac{3+\gamma}4}P_{T-t}w_{s,t}(x)+L^N_{s,t}(x),
\end{equation}
where $C=C(\gamma)$ and for any  $m\in[2,\infty)$ we have 
\begin{equation}\label{smallNbound}
\|L_{s,t}^N(x)\|_{\Lm}\le C\|f\|_{\C^\gamma}|t-s|^{1+\frac\gamma{4}},
\end{equation}
where $C=C(\gamma,m)$. 
Define $L_{s,t}(x):=\lim\inf_{N\to\infty}L^N_{s,t}(x)$.
To get \eqref{Boundrc} and \eqref{Boundlst}, we pass now to the limit as $N\to\infty$ in \eqref{bigNbound} and use Fatou's lemma and \eqref{smallNbound}.
\end{proof}

We also need the following integral bound obtained in \cite{ABLM24}.
\begin{proposition}\label{L:mainml}
	Let $\tau\in(0,1]$, $\nu\in[0,\frac12)$,  $0<T\le1$, $\gamma\in(-2,0)$, $\Gamma>0$. Let $f\colon\R\to\R$, $X\colon[0,T)\times D\to\R$ be measurable functions,   and let $\psi\colon[0,T]\times D\times\Omega\to\R$ be a measurable function adapted to the filtration $(\F_t)_{t\in[0,T]}$. Assume that 
	\begin{equation}\label{con.k}
		\gamma+4\tau>1;\quad \int_\0 |X_t(y)|dy\le \Gamma(T-t)^{-\nu},\quad t\in[0,T).
	\end{equation}
Suppose further that $f$ is bounded   or $f\ge0$.

Then for any $m\ge2$ there exists a constant $C=C(\gamma,\nu,\tau, m)>0$ such that for any $(s,t)\in\Delta_{[0,T]}$ 
	\begin{align}\label{Bound1K}
		&\Bigl\|\int_s^t\int_\0 X_r(y)f(V_r(y)+\psi_r(y))\,dydr\Bigr\|_{\Lm}\nn\\
		&\quad\le  C\|f\|_{\C^\gamma}\Gamma(t-s)^{1+\frac\gamma4-\nu}\bigl(1+    [\psi]_{\Ctimespace{\tau}{m}{[s,t]}}(t-s)^{\tau-\frac14}\bigr).
	\end{align} 
\end{proposition}

\begin{proof} If the function $f$ is continuous, this statement is \cite[Lemma~6.3]{ABLM24}. The extension  to bounded or nonnegative functions is done exactly as in  the proof of \cref{L:metaml} above, see also \cite[proof~of~Lemma~4.6]{BLM23}.
\end{proof}

We would often use the following simple corollary of the above bound.

\begin{corollary}\label{c:mainml}
Let $\gamma\in(-\frac32,-1)$, $p\in[1,\infty]$, let $g,h\colon\R\to\R$, $X\colon[0,1]\times D\to\R$ be measurable functions,  and let $u\colon[0,T]\times D\times\Omega\to\R$ be a measurable function adapted to the filtration $(\F_t)_{t\in[0,T]}$. 
Suppose further that $g$ is bounded and $h\in L_p(\R)$.
	
Then there exists a constant $C=C(\gamma,p)>0$ such that for any $t\in [0,1]$
	\begin{equation}\label{cBound1K}
		\Bigl\|\int_0^t\int_\0 X_r(y)(h-g)(u_r(y))\,dydr\Bigr\|_{L_2(\Omega)}\le  C\|h-g\|_{\C^{\gamma}}(1+[u-V]_{\C^{\frac58}L_2([0,1])})\sup_{r\in[0,1]}\|X_r\|_{L_1(D)}.
	\end{equation} 
\end{corollary}
\begin{proof}
Fix arbitrary $N\ge0$ and apply   \cref{L:mainml} twice with $f=g-h\I_{|h|\le N}$ and then with $f=|h|\I_{|h|\ge N}$. The remaining parameters are as follows: $\tau=\frac58$, $T=1$,  $\nu=0$, $\psi=u-V$.  We see that $\gamma+4\tau>1$ and condition \eqref{con.k} holds with $\Gamma=\sup_{r\in[0,1]}\|X_r\|_{L_1(D)}$. The function $g-h\I_{|h|\le N}$ is bounded and the function $|h|\I_{|h|\ge N}$ is nonnegative. Thus, all the assumptions of \cref{L:mainml} are satisfied and we deduce from \eqref{Bound1K}
\begin{align}\label{usefulcorol}
	&\Bigl\|\int_0^t\int_\0 X_r(y)(h-g)(u_r(y))\,dy\,dr\Bigr\|_{L_2(\Omega)}\nn\\
	&\qquad\le \Bigl\|\int_0^t\int_\0 X_r(y) (h\I_{|h|\le N}-g)(u_r(y))\,dy\,dr\Bigr\|_{L_2(\Omega)}\nn\\
	&\qquad\phantom{\le}+\Bigl\|\int_0^t\int_\0 X_r(y) |h|(u_r(y))\I_{|h|(u_r(y))\ge N}\,dy\,dr\Bigr\|_{L_2(\Omega)}\nn\\
	&\qquad\le C\bigl(\|h\I_{|h|\le N}-g\|_{\C^{\gamma}}+\|h\I_{|h|\ge N}\|_{L_p(\R)}\bigr)(1+[u-V]_{\C^{\frac58}L_2([0,1])})\sup_{r\in[0,1]}\|X_r\|_{L_1(D)}\nn\\
	&\qquad\le C\bigl(\|h-g\|_{\C^{\gamma}}+2\|h\I_{|h|\ge N}\|_{L_p(\R)}\bigr)(1+[u-V]_{\C^{\frac58}L_2([0,1])})\sup_{r\in[0,1]}\|X_r\|_{L_1(D)},
\end{align}
for $C=C(\gamma,p)$ independent of $N$, where we used the embedding $L_p(\R)\subset \C^{\gamma}$, valid since $\gamma<-1$. Since $h\in L_p(\R)$, we have $\|h\I_{|h|\ge N}\|_{L_p(\R)}\to0$ as $N\to\infty$. Therefore, by passing to the limit in \eqref{usefulcorol} as $N\to\infty$, we get  \eqref{cBound1K}.
\end{proof}

Finally, we provide here a corollary of the Kolomogorov continuity theorem which will be used in the paper.

\begin{proposition}\label{p:kolmi}
Let $\gamma_1, \gamma_2>0$.
Let $X^n\colon [0,1]\times D\times \Omega \to\R$, $n\in\Z_+$ be a sequence of jointly continuous processes such that $X^n(0,x)=0$ for any $x\in D$, $n\in\Z_+$.  Assume further that for any $m\ge1$ there exists a constant $C=C(m)$ such that the following bounds holds for any $s,t\in[0,1]$, $x,y\in D$, $n\in\Z_+$:
	\begin{align}
		&\|X^n(t,x)-X^n(s,y)\|_{L_m(\Omega)}\le C |t-s|^{\gamma_1}+C|x-y|^{\gamma_2};\label{con1k}\\
		&\lim_{n,k\to\infty}\sup_{t\in[0,1]}\sup_{x\in D }\|X^n(t,x)-X^k(t,x)\|_{L_m(\Omega)}=0.\label{con2k}
	\end{align}

Then there exists  a jointly continuous process $X\colon[0,1]\times D\times \Omega \to\R$ such that for any $N>0$
\begin{equation}\label{convkolm}
\sup_{s,t\in[0,1]}\sup_\DN |X(t,x)- X^{n}(t,x)|\to0 \quad\text{in probability as $n\to\infty$.} 
\end{equation}
Furthermore, for any $\rho\in(0,1)$, $\lambda>\gamma_2-\gamma_2\rho$ one has 
\begin{equation}\label{holderkolm}
	\sup_{s,t\in[0,1]}\sup_{x\in D} \frac{|X(t,x)- X(s,y)|}{(|t-s|^{\gamma_1}+|x-y|^{\gamma_2})^\rho(1+|x|+|y|)^\lambda}<\infty.
\end{equation}
\end{proposition}
\begin{proof}
Existence of a jointly continuous process $X$ satisfying \eqref{holderkolm} follows from, e.g., \cite[Proposition~C.1]{BLM23}. By passing to the limit as $n\to\infty$ in \eqref{con1k}, we get by Fatou's lemma
\begin{equation*}
\|X(t,x)-X(s,y)\|_{L_m(\Omega)}\le C |t-s|^{\gamma_1}+C|x-y|^{\gamma_2}.
\end{equation*}
Inequality \eqref{holderkolm}, which is a bound on a global weighted H\"older norm of $X$, follows now from \cite[Theorem~1.4.1]{Kunita} and a rescaling argument as in \cite[Lemmas~A.3-A.4 and Corollary A.5]{GG22}.
\end{proof}

\section{Proofs of main results}\label{s:4}

To simplify the notation, without loss of generality, we assume that the time interval is $[0,1]$. 
Denote for $n\in \N$, $t\in[0,1]$
\begin{equation*}
\kappa_n(t):=\lfloor nt\rfloor n^{-1}.
\end{equation*}
Thus, $\kappa_n(t)$ is the gridpoint of the grid $\{0,1/n,\hdots,1\}$ closest to $t$ from the left. 


For a function $\phi\in\S(D)$ introduce an increasing sequences of Hilbertian seminorms:
\begin{equation}\label{seminorm}
	\|\phi\|_\se:=\sum_{i=0}^m \int_{D}(1+|x|^m)^2 \Bigl|\frac{\d^i}{\d x^i}\phi(x)\Bigr|^2\,dx,\quad m\in\Z_+.
\end{equation}	
It is well-known and easy to check that these seminorms determine the usual Schwartz topology, see, e.g., \cite[Remark~1.3.1]{ito}.

\subsection{Taking the limit in the integral bounds}

Let us recall that the class $\V(\kappa)$ consists of functions that can be decomposed as the sum of a stochastic convolution $V$ and a sufficiently smooth perturbation. We show that limits similar to \eqref{idmildreg} and \eqref{idweakreg} hold for any function in this class. Note that, at this stage, we do not assume $u$ to be a solution to any equation; it is simply considered as a perturbation of $V$.

\begin{lemma}\label{l:43}
Let $\gamma\in(-2,0)$, $\tau\in(0,1)$. Let $f\in\C^\gamma$ and let $u\colon(0,1]\times D\times \Omega\to\R$ be a jointly continuous adapted process. Suppose that $u\in \V(\tau)$ and 
\begin{equation}\label{gammatau}
	\gamma+4\tau>1.
\end{equation}	
 Then the following holds:
\begin{enumerate}[(i)]
	\item there exists a jointly continuous adapted process $H^f\colon[0,1]\times D\times \Omega\to\R$ such that
	for any  sequence of functions $(f^n)_{n\in\Z_+}$ in $\C_b^\infty$ with $f^n\to f$ in $\C^{\gamma-}$ we have for any $N>0$
	\begin{equation}\label{firstflim}
	 \sup_{t\in[0,1]}\sup_\DN
	\lt|\int_0^t\int_\0 p_{t-r}(x,y) f^n(u_r(y))\,dy\,dr-H^f_t(x)\rt|\to0\quad  \text{in probability as $n\to\infty$.}	
	\end{equation}
	Furthermore, for any $\eps>0$ we have 
	\begin{equation}\label{Hregural}
\sup_{s,t\in[0,1]}\sup_{x,y\in D}\frac{|H_t^f(x)-H_s^f(y)|}{(|t-s|^{\frac12-\eps}+|x-y|^{1-\eps})(1+|x|+|y|)^{2\eps}}<\infty\quad \text{a.s.}
	\end{equation}		
		
	\item for any test function $\phi\in\S(D)$ there exists a continuous process $\H^f(\phi)\colon[0,1]\times \Omega\to\R$ such that for any  sequence of functions $(f^n)_{n\in\Z_+}$ in $\C_b^\infty$ with $f^n\to f$ in $\C^{\gamma-}$ we have
	\begin{equation}\label{secondflim}
		\sup_{t\in[0,1]}
		\lt|\int_0^t\int_\0 \phi(x) f^n(u_r(x))\,dx\,dr-\H^f_t(\phi)\rt|\to0\quad  \text{in probability as $n\to\infty$.}
	\end{equation}
\end{enumerate}	
\end{lemma}
\begin{proof} (i). Let $(f^n)_{n\in\Z_+}$ be any sequence of $\C_b^\infty$ functions   such that $f^n\to f$ in $\C^{\gamma-}$. 
We apply the Kolomogorov continuity theorem in the form of \cref{p:kolmi} to the sequence of processes
\begin{equation*}
H^{f,n}_t(x):=\int_0^t\int_\0 p_{t-r}(x,y) f^n(u_r(y))\,dy\,dr,\quad t\in[0,1],\,x\in D,
\end{equation*}	
where $n\in\Z_+$. Let us verify that conditions \eqref{con1k} and \eqref{con2k} of the theorem are satisfied. Let $n\in\Z_+$, $(s,t)\in \Delta_{[0,1]}$, $x,y\in D$. Then for any $m\ge2$ we get
\begin{align}\label{hdiffm}
\|H^{f,n}_t(x)-H^{f,n}_s(y)\|_{L_m(\Omega)}&\le\Bigl\|\int_s^t\int_\0 p_{t-r}(x,z) f^n(u_r(z))\,dz\,dr\Bigr\|_{L_m(\Omega)}\nn\\
&\phantom{\le}+\Bigl\|\int_0^s\int_\0 (p_{t-r}(x,z)-p_{s-r}(y,z)) f^n(u_r(z))\,dz\,dr\Bigr\|_{L_m(\Omega)}\nn\\
&=:I_1^m+I_2^m.
\end{align}
Fix $\gamma'<\gamma$  such that 
\begin{equation*}
\gamma'>-2\qquad \text{and}\qquad \gamma'+4\tau>1;
\end{equation*}
this is possible by assumptions of the lemma. We apply now \cref{L:mainml} to $X_r:=p_{t-r}(x,\cdot)$, $\nu=0$, $\Gamma=1$, $\psi=u-V$,  $\gamma'$ in place of $\gamma$ and $f^n$ in place of $f$. We note that the function $f^n$ is bounded by assumption and thus all the conditions of  \cref{L:mainml} are satisfied. Therefore, inequality \eqref{Bound1K} implies
\begin{equation}\label{ionem}
I_1^m\le C \|f^n\|_{\C^{\gamma'}}(t-s)^{1+\frac{\gamma'}4}(1+[u-V]_{\C^\tau L_m([0,1])})
\end{equation}
for $C=C(\gamma,\tau,m)>0$.
In a similar manner,  we apply \cref{L:mainml} to $T=s$, $X_r:=p_{t-r}(x,\cdot)-p_{s-r}(y,\cdot)$, where $r\in[0,s]$,$\psi=u-V$, $\gamma'$ in place of $\gamma$ and $f^n$ in place of $f$. We use \cite[Lemma~C.2]{ABLM24} to get for any $\nu\in(0,1/2)$, $r\in[0,s)$
\begin{align*}
\int_D |p_{t-r}(x,z)-p_{s-r}(y,z)|\,dz &\le \int_D (|p_{t-r}(x,z)-p_{t-r}(y,z)|+|p_{t-r}(y,z)-p_{s-r}(y,z)|)\,dz\\
&\le (|x-y|^{2\nu}+(t-s)^\nu) (s-r)^{-\nu}.
\end{align*}	
Thus \eqref{con.k} holds with $\Gamma:=|x-y|^{2\nu}+(t-s)^\nu$. Hence we get 
\begin{equation*}
	I_2^m\le C \|f^n\|_{\C^{\gamma'}} (|x-y|^{2\nu}+(t-s)^\nu)(1+[u-V]_{\C^\tau L_m([0,1])}),
\end{equation*}
where $C=C(\gamma,\nu,\tau,m)>0$. 
Combining this with \eqref{ionem} and \eqref{hdiffm}, we finally obtain  
\begin{equation*}
\|H^{f,n}_t(x)-H^{f,n}_s(y)\|_{L_m(\Omega)}\le C \|f^n\|_{\C^{\gamma'}} (|x-y|^{2\nu}+(t-s)^\nu)(1+[u-V]_{\C^\tau L_m([0,1])}),
\end{equation*}
where we used that $\nu<1/2$ and $1+\frac{\gamma'}4>1/2$. This shows \eqref{con1k} since ${[u-V]_{\C^\tau L_m([0,1])}<\infty}$ thanks to the condition $u\in\V(\tau)$ and $\sup_{n\in\Z_+}\|f^n\|_{\C^{\gamma'}}<\infty$ by the definition of convergence in $\C^{\gamma-}$.

To show \eqref{con2k}, we write for any $n,k\in\Z_+$, $t\in[0,1]$, $x\in D$
\begin{align*}
	\|H^{f,n}_t(x)-H^{f,k}_t(x)\|_{L_m(\Omega)}&\le\Bigl\|\int_0^t\int_\0 p_{t-r}(x,z) (f^n-f^k)(u_r(z))\,dz\,dr\Bigr\|_{L_m(\Omega)}\\
	&\le  C \|f^n-f^k\|_{\C^{\gamma'}} (1+[u-V]_{\C^\tau L_m([0,1])}),
\end{align*}
where the last inequality follows from \cref{L:mainml}. Since as above  $u\in\V(\tau)$ implies  $[u-V]_{\C^\tau L_m([0,1])}<\infty$ and  $\lim_{n,k\to\infty}\|f^n-f^k\|_{\C^{\gamma'}}=0$ by the definition of convergence in $\C^{\gamma-}$, we see that condition \eqref{con2k} holds. Thus, all the conditions of \cref{p:kolmi} are satisfied. Therefore, \eqref{convkolm} implies \eqref{firstflim} and \eqref{holderkolm} yields \eqref{Hregural}.

It is immediate to see that the limiting process $H^f$ does not depend on the approximating sequence $(f^n)_{n\in\Z_+}$. Indeed, let $(\tilde f^{n})_{n\in\Z_+}$ be any other sequence of $\C_b^\infty$ functions  converging to $f$ in $\C^{\gamma-}$. Then for the merged sequence $(h^n)_{n\in\Z_+}$, given by $h^{2n}=f^n$ and $h^{2n+1}=\tilde f^n$, $n\in\Z_+$, we have by \eqref{firstflim} for any $N>0$
	\begin{equation*}
	\sup_{t\in[0,1]}\sup_\DN 
	\lt|\int_0^t\int_\0 p_{t-r}(x,y) h^n(u_r(y))\,dy\,dr-\wt H^f_t(x)\rt|\to0\quad  \text{in probability as $n\to\infty$;}
\end{equation*}
for a certain jointly continuous adapted process $\wt H^f\colon[0,1]\times D\times \Omega\to\R$. However, recalling \eqref{firstflim}, we see that $\int_0^t\int_\0 p_{t-r}(x,y) h^{2n}(u_r(y))\,dy\,dr$ converges in probability as $n\to\infty$ uniformly on compacts to $H^f_t(x)$. Hence a.s. we have $H^f_t(x)=\wt H^f_t(x)$ for all $t\in[0,1]$, $x\in D$. Thus, 
$$
\int_0^t\int_\0 p_{t-r}(x,y) \wt f^{n}(u_r(y))\,dy\,dr=\int_0^t\int_\0 p_{t-r}(x,y) h^{2n+1}(u_r(y))\,dy\,dr\to H^f_t(x)
$$
 in probability as $n\to\infty$ uniformly on compacts. Therefore, the limiting process $H^f$ does not depend on the approximating sequence.

(ii). We argue similarly to part (i) of the proof. Fix $\phi\in\S(D)$, $m\ge2$. We would like to apply \cref{p:kolmi} to the sequence of processes
\begin{equation*}
\H^{f,n}_t:=\int_0^t\int_\0 \phi(x)  f^n(u_r(x))\,dx\,dr,\quad t\in[0,1],
\end{equation*}
where $n\in\Z_+$. First, we verify that this sequence satisfies condition  \eqref{con1k}.
We get
\begin{align*}
	\|\H^{f,n}_t-\H^{f,n}_s\|_{L_m(\Omega)}&\le\Bigl\|\int_s^t\int_\0 \phi(x)  f^n(u_r(x))\,dx\,dr\Bigr\|_{L_m(\Omega)}\\
	&\le  C \|f^n\|_{\C^\gamma}\|\phi\|_{L_\infty(D)}(t-s)^{1+\frac\gamma4}(1+[u-V]_{\C^\tau L_m([0,1])}), 
\end{align*}
where the last inequality follows from \cref{L:mainml} applied to $X:=\phi$, $f^n$ in place of $f$, $\nu=0$. Since $u\in\V(\tau)$, we see that $[u-V]_{\C^\tau L_m([0,1])}<\infty$ and thus  \eqref{con1k} holds.

Next, to show  \eqref{con2k}, we  pick as above any $\gamma'<\gamma$ such that $\gamma'+4\tau>1$. Then by \cref{L:mainml}, we get 
\begin{align*}
	\|\H^{f,n}_t-\H^{f,k}_t\|_{L_m(\Omega)}&\le\Bigl\|\int_0^t\int_\0 \phi(x)  (f^n-f^k)(u_r(z))\,dz\,dr\Bigr\|_{L_m(\Omega)}\\
	&\le  C \|f^n-f^k\|_{\C^{\gamma'}} (1+[u-V]_{\C^\tau L_m([0,1])}).
\end{align*}
We note that as above $[u-V]_{\C^\tau L_m([0,1])}<\infty$ thanks to the assumption $u\in\V(\tau)$ and thus \eqref{con2k} holds.

Thus, we have verified all the conditions of \cref{p:kolmi}. The desired limit \eqref{secondflim} follows now from \eqref{convkolm}. Arguing exactly as in the end of part (i) of the proof, we see that the limiting process $\H^f(\phi)$ does not depend on the approximating sequence $(f^n)_{n\in\Z_+}$.
\end{proof}	

\cref{l:43} allows us to show that any mild regularized solution to \eqref{SPDE} cannot grow too fast as $|x|\to\infty$. This will be important later to prove that if $u$ is a  mild regularized solution to \eqref{SPDE}, then the process $t\mapsto \la u_t,\phi\ra$ is a.s. continuous for any test function $\phi\in\S(D)$.
\begin{corollary}\label{c:globalspaceb}Let $\alpha>-3/2$, $b\in\C^\alpha$, 
$u_0\in\Bsp(\0)$. Let 
\begin{equation*}
 u_t(x)=P_tu_0(x)+K_t(x)+V_t(x),\quad  x\in\0,\,\, t\in(0,1]
\end{equation*}
be a mild regularized solution  of equation  \eqref{SPDE}. Suppose that $u\in\V(\frac58)$. Then
for any $\eps>0$
\begin{align}
\label{diffbound}
&\sup_{s,t\in[0,1]}\sup_{x,y\in D}\frac{|K_t(x)-K_s(y)|}{(|t-s|^{\frac12-\eps}+|x-y|^{1-\eps})(1+|x|+|y|)^{2\eps}}<\infty\quad \text{a.s.;}	
\\
\label{supbound}
&\sup_{t\in[0,1]}\sup_{x\in D}\frac{|u_t(x)|}{1+|x|^{\eps}}<\infty\quad \text{a.s.}	
\end{align}
\end{corollary}
\begin{proof}
We apply \cref{l:43}. Let $(b^n)_{n\in\Z_+}$ be a sequence of functions in $\C_b^\infty$ converging to $b$ in $\C^{\alpha-}$.  Recall that by definition of regularized solution we have
\begin{equation*}
	 \sup_{t\in[0,1]}\sup_\DN
	\lt|\int_0^t\int_\0 p_{t-r}(x,y) b^n(u_r(y))\,dy\,dr-K_t(x)\rt|\to0\quad  \text{in probability as $n\to\infty$.}	
\end{equation*}
Therefore, the process $K$ coincides with the process $H^b$ define in \eqref{firstflim} with $b$ in place of $f$. Inequality \eqref{diffbound} follows now from \eqref{Hregural}.

It follows immediately by a standard direct calculation (see, e.g., \cite[Proposition 3.2.2]{Marta}) that for any $m\ge1$ there exists a constant $C=C(m)$ such that for any $(s,t)\in\Delta_{[0,1]}$, $x,y\in D$ we have
\begin{equation*}
\|V(t,x)-V(s,y)\|_{L_m(\Omega)}\le C |t-s|^{\frac14}+C|x-y|^{\frac12}.
\end{equation*}	
Applying the Kolmogorov continuity theorem, we deduce for any $\eps>0$.
\begin{equation}\label{vvvbound}
\sup_{t\in[0,1]}\sup_{x\in D}\frac{|V(t,x)|}{1+|x|^{\eps}}<\infty\quad \text{a.s.}		
\end{equation}
Recall that $\sup_{t\in[0,1],x\in D}|P_tu_0(x)|<\infty$, since $u_0$ is bounded. Therefore, 
using \eqref{diffbound} with $s=0$, $y=x$ and \eqref{vvvbound}, we get for any $\eps>0$ 
\begin{equation*}
\sup_{t\in[0,1]}\sup_{x\in D}\frac{|K_t(x)|+|P_tu_0(x)|+|V(t,x)|}{1+|x|^{\eps}}<\infty\quad \text{a.s.},
\end{equation*}
which is \eqref{supbound}.
\end{proof}

\subsection{From weak to mild formulation}

Let us recall that our goal is to pass from the identities \eqref{idweak} and \eqref{idweakreg}, which appear in the definitions of analytically weak and regularized weak solutions, respectively, to the identities \eqref{SPDEfunc} and \eqref{idmildid}. In this subsection, we take the first step in this direction, which is common to both analytically weak and regularized weak solutions.

We begin by proving that \eqref{idweak} and \eqref{idweakreg} can be extended to accommodate test functions $\phi$ that depend on time. Recall the definition of a seminorm$\|\cdot\|_\se$ in \eqref{seminorm}.

\begin{lemma}	\label{l:phitd}
Let $u_0\in\Bsp(\0)$ and let $u\colon(0,1]\times\0\times\Omega\to\R$ be a jointly continuous  process.
Let $\H\colon[0,1]\times \S(D)\times \Omega\to\R$ be a continuous in time process. Assume that 
for any $\phi\in\S(D)$ the following identity holds $\P$-a.s.
\begin{equation}\label{uid}
	\langle u_t,\phi\rangle= \langle u_0,\phi\rangle +\int_0^t \langle u_s, \frac{1}{2}\Delta\phi\rangle\,ds  +\H_t(\phi)+W_t(\phi),\quad t\in[0,1]
\end{equation}
and all the integrals are  finite. Then the following holds.
\begin{enumerate}[(i)] 
\item There exists a set $\Omega'\subset\Omega$ of full probability measure with the following property: for any $\omega\in\Omega'$ there exist $m=m(\omega)\in\Z_+$, $\Gamma=\Gamma(\omega)>0$ such that for any $\phi\in\S(D)$
\begin{equation}\label{mainsemi}
\sup_{t\in[0,1]}|\la u_t(\omega),\phi\ra|\le \Gamma(\omega)\|\phi\|_{\se(\omega)}.
\end{equation}

\item For any $t\in(0,1]$, $x\in D$ we have 
\begin{equation}
P_{\eps}u_{t-\eps}(x)\to u_t(x)\quad\text{a.s. as $\eps\to0$.}\label{Pconv}
\end{equation}
\item 
Let 
$f\in\C([0,1],\S(D))$.
Then there exists a continuous adapted process denoted as $\int_0^\cdot \H(ds,f_s)\colon [0,1]\times \Omega\to\R$ such that 
 for any $t\in[0,1]$ we have
\begin{equation}\label{eq:13_10_5}
\sum_{i=1}^{\lfloor nt\rfloor} (\H_{\frac{i}n}(f_{\frac{i-1}n})-\H_{\frac{i-1}n}(f_{\frac{i-1}n}))
			\to\int_0^t \H(ds,f_s)
\end{equation}
in probability as $n\to\infty$. 
Further,  a.s. for any $t\in[0,1]$
\begin{equation}\label{identl41}
\langle u_t, f_t\rangle= \langle u_{0},f_{0}\rangle + \int_{0}^t \langle u_s, \frac{\partial  f_s}{\partial s}+ \frac{1}{2}\Delta f_s\rangle\,ds  +\int_0^t \H(ds,f_s)+\int_{0}^t\int_D f_{s}(x)W(ds,dx).
\end{equation}
\end{enumerate}

\end{lemma}

\begin{proof}
(i). We see that thanks to \eqref{uid}  
\begin{equation}\label{cont}
	\text{for any $\phi\in\S(D)$ the process $r\mapsto \langle u_r,\phi\rangle$ is a.s. continuous}.
\end{equation}	
Then applying Mitoma's theorem \cite[Theorem~1 and formula (2.12)]{mitoma}  to the Fr\'eche space $\S(D)$, we get precisely  \eqref{mainsemi}.

(ii). Fix  $x\in D$, $t\in(0,1]$. Let $\chi\colon\R\to[0,1]$ be a $\C^\infty(\R)$ function such that $\chi(y)=1$ if $|y|\le1$, $\chi(y)=0$ if $|y|\ge2$, and $\chi(y)\in(0,1)$ for $|y|\in(1,2)$. Then 
\begin{align}\label{step1102}
P_\eps u_{t-\eps}(x)&=\la u_{t-\eps},p_\eps(x,\cdot)\ra =\la u_{t-\eps},p_\eps(x,\cdot)\chi(x-\cdot)\ra+\bigl\la u_{t-\eps},p_\eps(x,\cdot)(1-\chi(x-\cdot))\bigr\ra\nn\\
&=:J_1(\eps)+J_2(\eps).
\end{align}
Since $u$ is a.s.  continuous on $[t/2,t]\times([x-2,x+2]\cap D)$, it is easy to see that 
\begin{equation}\label{step111}
J_1(\eps)\to u_t(x),\quad\text{a.s.  as $\eps\to0$}. 
\end{equation}
To pass to the limit in $J_2(\eps)$, we use part (i) of the lemma. It follows from \eqref{mainsemi} that if $\omega\in\Omega'$ then for some $m(\omega)$, $\Gamma(\omega)$ we have 
\begin{equation*}
|J_2(\eps,\omega)|\le \Gamma(\omega)\|p_\eps(x,\cdot)(1-\chi(x-\cdot))\|_{\se(\omega)}\to 0,\quad\text{as $\eps\to0$}. 
\end{equation*}
Combining this with \eqref{step1102} and \eqref{step111}, we get \eqref{Pconv}.

(iii). Fix $f\in\C([0,1],\S(D))$. Let $n\in \N$. Using identity \eqref{uid}, we clearly have for any $k=1,..., n$ a.s.
\begin{align*}
\langle u_{\frac{k}n}, f_{\frac{k}n}\rangle &=\langle u_{\frac{k}n},f_{\frac{k}n}-f_{\frac{k-1}n}\rangle+\langle u_{\frac{k}n},f_{\frac{k-1}n}\rangle\\
&=\langle u_{\frac{k}n},f_{\frac{k}n}-f_{\frac{k-1}n}\rangle+\langle u_{\frac{k-1}n}, f_{\frac{k-1}n}\rangle 
	+\int_{\frac{k-1}n}^{\frac{k}n} \langle u_s, \frac{1}{2}\Delta f_{\frac{k-1}n}\rangle\,ds  +\H_{\frac{k}n}(f_{\frac{k-1}n})-\H_{\frac{k-1}n}(f_{\frac{k-1}n})\\
	&\phantom{=}+W_{\frac{k}n}(f_{\frac{k-1}n})-W_{\frac{k-1}n}(f_{\frac{k-1}n}).
\end{align*}
Summing these identities over $k$ we derive for any $t\in[0,1]$
\begin{align}\label{ident}
	\langle u_{\kappa_n(t)}, f_t\rangle&=\langle u_{\kappa_n(t)}, f_{\kappa_n(t)}\rangle+\langle u_{\kappa_n(t)}, f_t-f_{\kappa_n(t)}\rangle\nn\\
	&= \langle u_{0},f_{0}\rangle +\langle u_{\kappa_n(t)}, f_t-f_{\kappa_n(t)}\rangle+ \sum_{i=1}^{\lfloor nt\rfloor} 
		\langle u_{\frac{i}{n}},f_{\frac{i}{n}}-f_{\frac{i-1}{n}}\rangle
		+\int_{0}^{\kappa_n(t)} \langle u_s, \frac{1}{2}\Delta f_{\kappa_n(s)}\rangle\,ds\nn\\
		&\phantom{=}+\sum_{i=1}^{\lfloor nt\rfloor} (\H_{\frac{i}n}(f_{\frac{i-1}n})-\H_{\frac{i-1}n}(f_{\frac{i-1}n}))+\int_0^{\kappa_n(t)}\int_D f_{\kappa_n(s)}(x)\,W(ds,dx).
\end{align}

Fix now $t\in[0,1]$ and let us pass to the limit as $n\to\infty$ in this identity.  Using \eqref{cont} with $\phi=f_t\in\S(D)$, we get 
\begin{equation}\label{lemlim1}
\langle u_{\kappa_n(t)}, f_t\rangle\to   \langle u_t, f_t\rangle,\quad\text{a.s. as $n\to\infty$.} 
\end{equation}	
Next, we note 
\begin{equation}\label{limbefore}
\langle u_{\kappa_n(t)}, f_t-f_{\kappa_n(t)}\rangle+ \sum_{i=1}^{\lfloor nt\rfloor} 
		\langle u_{\frac{i}{n}},f_{\frac{i}{n}}-f_{\frac{i-1}{n}}\rangle=\int_0^{t} \langle u_{(\kappa_n(s)+\frac1n)\wedge \kappa_n(t)}, \frac{\d f_s}{\d s}\rangle\,ds.
\end{equation}
We note that  by \eqref{cont} for any $s\in[0,t]$ we have $\langle u_{(\kappa_n(s)+\frac1n)\wedge \kappa_n(t)}, \frac{\d f_s}{\d s}\rangle\to\langle u_{s}, \frac{\d f_s}{\d s}\rangle$  a.s. as $n\to\infty$. By part (i) of the lemma, we get
\begin{equation*}
\sup_{n\in\N}\sup_{r\in[0,t]}|\langle u_{(\kappa_n(r)+\frac1n)\wedge \kappa_n(t)}, \frac{\d f_r}{\d r}\rangle|<\infty\quad a.s.	
\end{equation*}
Hence by the Fubini theorem there exists a set $\Omega''\subset\Omega$ of full probability measure such that if $\omega\in\Omega''$, then for Lebesgue almost all $s\in[0,t]$ we have 
\begin{equation*}
\lim_{n\to\infty}\langle u_{(\kappa_n(s)+\frac1n)\wedge \kappa_n(t)}(\omega), \frac{\d f_s}{\d s}\rangle=\langle u_{s}, \frac{\d f_s}{\d s}\rangle\quad \text{and} \quad
	\sup_{n\in\N}\sup_{r\in[0,t]}|\langle u_{(\kappa_n(r)+\frac1n)\wedge \kappa_n(t)}(\omega), \frac{\d f_r}{\d r}\rangle|<\infty.	
\end{equation*}
Thus, by the dominated convergence theorem we can pass to the limit in \eqref{limbefore} and get 
\begin{equation}\label{lemlim2}
\int_0^{t} \langle u_{(\kappa_n(s)+\frac1n)\wedge \kappa_n(t)}, \frac{\d f_s}{\d s}\rangle\,ds\to \int_0^{t} \langle u_{s}, \frac{\d f_s}{\d s}\rangle\,ds,\quad\text{a.s. as $n\to\infty$.}
\end{equation}
For the next term in \eqref{ident} we have 
\begin{equation}\label{lemlim3}
\int_{0}^{\kappa_n(t)} \langle u_s, \frac{1}{2}\Delta f_{\kappa_n(s)}\rangle\,ds=
\int_{0}^t \langle u_s, \frac{1}{2}\Delta f_{\kappa_n(s)}\rangle\,ds-
\int_{\kappa_n(t)}^t \langle u_s, \frac{1}{2}\Delta f_{\kappa_n(s)}\rangle\,ds\to
\int_{0}^{t} \langle u_s, \frac{1}{2}\Delta f_s\rangle\,ds
\end{equation}
a.s. as $n\to\infty$. Here we used that thanks to part (i) of the lemma and the continuity of $f$ in $\S(D)$ topology 
\begin{equation*}
\sup_{s\in[0,t]}|\langle u_s,\Delta f_{\kappa_n(s)}-\Delta f_s\ra|\to0\,\,\,a.s.\quad\text{and}\quad \sup_{s,r\in[0,t]}|\langle u_s,\Delta f_r\ra|<\infty\,\,\,a.s.
\end{equation*}

Finally, using the Ito isometry, it is  easy to see that for any $\delta>0$
\begin{equation*}
\lim_{n\to\infty } \P\Bigl(\Bigl|\int_{0}^{\kappa_n(t)}\int_D f_{\kappa_n(s)}(x)W(ds,dx)- \int_{0}^t\int_D f_{s}(x)W(ds,dx)\Bigr|>\delta\Bigr)=0.
\end{equation*}
Combining this with \eqref{lemlim1}, \eqref{lemlim2}, \eqref{lemlim3} and substituting into \eqref{ident}, we see that for each $f$ there exists a measurable process $\int_0^\cdot \H(ds,f_s)\colon [0,1]\times \Omega\to\R$ such that \eqref{eq:13_10_5} holds and 
and for $t\in[0,1]$
\begin{equation*}
\langle u_t, f_t\rangle= \langle u_{0},f_{0}\rangle + \int_{0}^t \langle u_s, \frac{\partial  f_s}{\partial s}\rangle\,ds
		+\int_{0}^t \langle u_s, \frac{1}{2}\Delta f_s\rangle\,ds  +\int_0^t \H(ds,f_s)+\int_{0}^t\int_D f_{s}(x)W(ds,dx).
\end{equation*}
This concludes the proof of the lemma.
\end{proof}

Now we are ready to show that if  $u$ satisfies a ``weak-type'' identity \eqref{uid}, then it also satisfies a ``mild-type'' identity. 
\begin{lemma}
\label{l:txrepr}
Assume that the conditions of \cref{l:phitd} are satisfied. Then there exists a stochastic process $R_\H\colon[0,1]\times D\times\Omega\to\R$ such that
$\P$-a.s. $u$ satisfies the following equation 
\begin{equation}\label{eq:13_10_2}
	u_t(x)= P_tu_{0}(x)   +R_\H(t,x)+V_t(x),\quad  t\in[0,1],\,\, x\in D. 
\end{equation}
Further, $R_\H$ is a.s. continuous on $(0,1]\times D$  and for any fixed $t>0$, $x\in D$ we have 
\begin{equation}\label{eq:10_08_2}
\int_0^{t-\eps} \H(dr, p_{t-r}(x,\cdot))\to R_\H(t,x) 
\end{equation}
in probability as $\eps\to0$.
\end{lemma}
\begin{proof}
\textbf{Step~1}. Take arbitrary $\phi\in\S(D)$.
Fix $t\in[0,1]$ and apply \cref{l:phitd}(iii) with $f_s:=P_{t-s}\phi$, $s\in[0,t]$. Clearly,  $f\in\C([0,1],S(D))$.
Therefore all the conditions of \cref{l:phitd}(iii) are satisfied. Note that
\begin{equation*}
\frac{\partial  f_s}{\partial s}+  \frac{1}{2}\Delta f_s=0, \quad s\in[0,t].
\end{equation*}
Therefore, \cref{l:phitd}(iii) implies that there exists a process $\int_0^\cdot \H(dr, {P_{t-r}\phi})\colon[0,t]\times\Omega\to\R$, where the integral is defined in \eqref{eq:13_10_5}, such that 
\begin{equation*}
\langle u_s,P_{t-s}\phi\rangle= \langle u_{0},P_t \phi\rangle   +\int_0^s \H(dr, {P_{t-r}\phi})
+\int_{0}^s\int_D P_{t-r}\phi(x)W(dr,dx),\quad  s\in [0,t].
\end{equation*}
By choosing $s=t$, and using  that $P_t$ is self adjoint we get  
\begin{equation}\label{eq:16_08_3}
\langle u_t,\phi\rangle= \langle P_t u_{0},\phi\rangle   +\int_0^t \H(dr, {P_{t-r}\phi})+\int_{0}^t\int_D P_{t-r}\phi(x)W(dr,dx),\quad  t\in [0,1].
\end{equation}

\textbf{Step 2}. For arbitrary fixed $x\in D$, $\eps>0$, take $\phi:=p_\eps(x,\cdot)$. Clearly, $\phi\in\S(D)$. Substituting this $\phi$ in~\eqref{eq:16_08_3} with $t-\eps$ in place of $t$, 
we get for $t\in[\eps,1]$
\begin{align}
P_{\eps} u_{t-\eps}(x)&= P_{t}u_0(x)  +\int_0^{t-\eps} \H(dr, p_{t-r}(x,\cdot))+\int_{0}^{t-\eps}\int_D p_{t-r}(x,y)W(dr,dy)\nn\\
&= P_{t}u_0(x)  +\int_0^{t-\eps} \H(dr, p_{t-r}(x,\cdot))+P_{\eps}V_{t-\eps}(x),\quad \P-a.s.\label{eq:16_08_4}
\end{align}
Denote 
\begin{equation*}
R_\H(t,x):=	u_t(x)-P_tu_{0}(x)-V_t(x),\quad x\in D,\,\, t>0.
\end{equation*}
Then, by definition \eqref{eq:13_10_2} holds. 
By \cref{l:phitd}(ii), $P_\eps u_{t-\eps}(x)\to u_t(x)$ a.s. as $\eps\to0$. 
We also have by Ito's isometry
$P_\eps V_{t-\eps}(x)\to V_t(x)$ in $L_2(\Omega)$  as $\eps\to0$. Thus, recalling  identity \eqref{eq:16_08_4}, we get \eqref{eq:10_08_2}.
\end{proof}

\subsection{Proof of \cref{thm:main0}: (a) implies (b)}
Let us recall that if $u$ is a weak regularized solution, then, by definition, identity \eqref{uid} holds with $\H = \K$, where $\K$ is defined in \cref{Def:weakrsol}. Therefore, $u$ satisfies identity \eqref{eq:13_10_2} for a process $R_\K$ defined in \eqref{eq:10_08_2}. Thus, it remains to show that $R_\K$ is, in fact, the drift term $K$ that appears in \cref{Def:mildrsol}. 

It follows from \cref{l:43}, that the limit in \eqref{idmildreg} holds for a certain stochastic process $H$; therefore, we have to show that $R_\K$ coincides with $H$. This is done in the following lemma. 
\begin{lemma}
\label{prop:10_11_1}
Let $\alpha>-3/2$, $b\in\C^\alpha$, 
$u_0\in\Bsp(\0)$. Let $u$ be a weak regularized solution  of equation  \eqref{SPDE} with the initial condition $u_0$ and suppose that $u\in\V(5/8)$.

Then for any  sequence of functions $(b^n)_{n\in\Z_+}$ in
      $\C_b^\infty$ such that $b^n\to b$ in $\C^{\alpha-}$, $t\in[0,1]$, $x\in D$ we have
\begin{equation}
\label{eq:3_11_1a}
\int_0^t\int_D p_{t-r}(x,y) b^n(u_r(y))\,dydr\to R_\K(t,x)\quad  \text{in probability as $n\to\infty$.}
\end{equation}
\end{lemma}

To obtain this result we fix $f\in\C_b^\infty$ and for $t\in[0,1]$, $x\in D$, 
$\eps\in(0,t)$ use the following decomposition
\begin{align}\label{eq:decompos}
&\Bigl|\int_0^t\int_D p_{t-r}(x,y) f(u_r(y))\,dy\,dr-R_\K(t,x)\Bigr|\nn\\
 &\qquad
\le  \Bigl|\int_0^{t-\eps}\int_D p_{t-r}(x,y) f(u_r(y))\,dy\,dr-
 \int_0^{t-\eps} \K(dr, p_{t-r}(x,\cdot))\Bigr|\nn\\
 &\qquad\phantom{\le}+ 
 \Bigl|\int_{t-\eps}^t\int_D p_{t-r}(x,y) f(u_r(y))\,dy\,dr\Bigr|+ 
\Bigl|R_\K(t,x)-\int_0^{t-\eps} \K(dr, p_{t-r}(x,\cdot)) \Bigr|\nn
\\&\qquad =:I_1^{\eps}(t,x)+I_2^{\eps}(t,x)+I_3^{\eps}(t,x), 
\end{align}
Here and later on, for a suitable test function $f\in\C([0,1]\times D,\R)$, the nonlinear integral $\int_0^{\cdot} \K(ds, f_s)$ is defined as in \eqref{eq:13_10_5} with $\H=\K$.

We begin with the analysis of  the first term in \eqref{eq:decompos}.

\begin{lemma}
\label{lem:3_11_1}
Assume that the assumptions of \cref{prop:10_11_1} holds. Let $\alpha'\in(-\frac32,\alpha)$. Then there exists a constant $C=C(\alpha,\alpha')$ such that 
for any $x\in D$, $t\in[0,1]$, $\eps\in(0,t)$, $f\in\C^\infty_b$, $\delta>0$ we have
\begin{equation}\label{desbound}
\P(I^{\eps}_{1}(t,x)>\delta)\le C \delta^{-1}(1+[u-V]_{\C^{\frac58}L_2([0,1])})\|f-b\|_{\C^{\alpha'}}.
\end{equation}
\end{lemma} 
\begin{proof} Fix $x\in D$, $t\in[0,1]$,  $\eps\in(0,t)$, $\delta>0$. 
 Take arbitrary $n\in \N$. We further decompose $I_1^{\eps}$ and write 
\begin{align}\label{decompivan}
I_1^{\eps}(t,x)&\le 
\Bigl|\int_0^{t-\eps}\int_D (p_{t-r}(x,y)-  p_{t-\kappa_n(r)}(x,y)) f(u_r(y))\,dy\,dr\Bigr|\nn\\
&\phantom{\le}+\Bigl|
\int_0^{t-\eps}\int_D p_{t-\kappa_n(r)}(x,y) f(u_r(y))\,dy\,dr-\sum_{i=1}^{\lfloor n(t-\eps)\rfloor} (\K_{\frac{i}n}(p_{t-\frac{i-1}n}(x,\cdot))-\K_{\frac{i-1}n}(p_{t-\frac{i-1}n})(x,\cdot))\Bigr|\nn\\
&\phantom{\le}+\Bigl|
\sum_{i=1}^{\lfloor n(t-\eps)\rfloor} (\K_{\frac{i}n}(p_{t-\frac{i-1}n}(x,\cdot))-\K_{\frac{i-1}n}(p_{t-\frac{i-1}n})(x,\cdot))-\int_0^{t-\eps} \K(dr, p_{t-r}(x,\cdot))\Bigr|\nn\\&
=:I^{n,\eps}_{1,1}(t,x)+I^{n,\eps}_{1,2}(t,x)+I^{n,\eps}_{1,3}(t,x). 
\end{align}

We bound $I^{n,\eps}_{1,1}$ using \cref{L:mainml} with $\psi:=u-V$, $T:=t-\eps$, $\gamma=\alpha$, $\tau=\frac58$, $X_r:=p_{t-r}(x,\cdot)-  p_{t-\kappa_n(r)}(x,\cdot)$, where $r\in[0,t-\eps]$, $\nu=0$. It is easy to see that  $\alpha+4\tau>1$ thanks to the standing assumption $\alpha>-3/2$. Further, for $r\in[0,t-\eps]$
\begin{equation*}
\int_D |p_{t-r}(x,y)-  p_{t-\kappa_n(r)}(x,y)|\,dy\le C (t-\kappa_n(r))^{-1/2}|r-\kappa_n(r)|\le Cn^{-1}\eps^{-\frac12},
\end{equation*}
for some $C>0$ independent of $n$, $x$, see, e.g. \cite[Lemma~C.2]{ABLM24}. Thus, condition \eqref{con.k} holds with $\Gamma=Cn^{-1}\eps^{-\frac12}$.  
Therefore all the assumptions of \cref{L:mainml} are satisfied and we get from \eqref{Bound1K}
\begin{equation*}
\|I^{n,\eps}_{1,1}(t,x)\|_{L_2(\Omega)}\le C \|f\|_{\C^\alpha} n^{-1}\eps^{-\frac12}(1+[u-V]_{\C^{\frac58}L_2([0,1])}),\quad t\in[\eps,1],\, x\in D,
\end{equation*}
for $C=C(\alpha)$. Since $u\in\V(\frac58)$, we have $[u-V]_{\C^{\frac58}L_2([0,1])}<\infty$. Hence, by the Chebyshev inequality,
\begin{equation}\label{ioned}
\lim_{n\to\infty}\P(I^{n,\eps}_{1,1}(t,x)>\delta)=0.
\end{equation}

Next, we bound $I^{n,\eps}_{1,2}$. Let $(b^M)_{M\in\Z_+}$ be a sequence of functions in
      $\C_b^\infty$ such that $b^M\to b$ in $\C^{\alpha-}$. By the definition of $\K$ in \eqref{Kcurlydef}. 
\begin{align*}
&\sum_{i=1}^{\lfloor n(t-\eps)\rfloor} (\K_{\frac{i}n}(p_{t-\frac{i-1}n}(x,\cdot))-\K_{\frac{i-1}n}(p_{t-\frac{i-1}n}(x,\cdot)))\\
&\qquad=\lim_{M\to\infty} \int_0^{\kappa_n(t-\eps)}\int_D p_{t-\kappa_n(r)}(x,y) b^M(u_r(y))\,dy\,dr,
\end{align*}
where the limit is in probability. Hence, Fatou's lemma implies
\begin{align}\label{itwon}
\|I^{n,\eps}_{1,2}(t,x)\|_{L_2(\Omega)}&\le \lim_{M\to\infty }\Bigl\|\int_0^{t-\eps}\int_D p_{t-\kappa_n(r)}(x,y) (f-b^M)(u_r(y))\,dy\,dr\Bigr\|_{L_2(\Omega)}\nn\\
&\phantom{\le}+
\Bigl\|\int_{\kappa_n(t-\eps)}^{t-\eps}\int_D p_{t-\kappa_n(r)}(x,y) f(u_r(y))\,dy\,dr\Bigr\|_{L_2(\Omega)}.
\end{align}
To bound both  terms in the right-hand side of \eqref{itwon} the above inequality we take any $\alpha'\in(-\frac32,\alpha)$ and use again \cref{L:mainml}  with $\psi:=u-V$, $T:=t-\eps$, $\gamma=\alpha'$, $\tau=\frac58$, $X_r:=p_{t-\kappa_n(r)}(x,\cdot)$, where $r\in[0,t-\eps]$, $\nu=0$. For the first integral term, we take $f-b^M$ in place of $f$ in the lemma; note that  $f-b^M$ is bounded.  Since 
\begin{equation*}
\int_D p_{t-\kappa_n(r)}(x,y)\,dy=1,
\end{equation*}
and $\alpha'+4\tau>1$, we see that \eqref{con.k} holds with $\Gamma=1$. Thus all the assumptions of \cref{L:mainml} are satisfied and we derive for any $t\in[\eps,1]$, $x\in D$
\begin{align*}
\|I^{n,\eps}_{1,2}(t,x)\|_{L_2(\Omega)}&\le  C (1+[u-V]_{\C^{\frac58}L_2([0,1])})\bigl(\lim_{M\to\infty }\|f-b^M\|_{\C^{\alpha'}}+ n^{-1/2}\|f\|_{\C^{\alpha'}}\bigr)\\
&\le  C (1+[u-V]_{\C^{\frac58}L_2([0,1])})\bigl(\|f-b\|_{\C^{\alpha'}}+ n^{-1/2}\|f\|_{\C^{\alpha}}\bigr),
\end{align*}
where $C=C(\alpha,\alpha')$ and we used that $b^M\to b$ in $\C^{\alpha'}$ because $\alpha'<\alpha$ and $b^M\to b$ in $\C^{\alpha-}$. Using again that  $[u-V]_{\C^{\frac58}L_2([0,1])}<\infty$, we get 
\begin{equation}\label{itwod}
\lim_{n\to\infty}\P(I^{n,\eps}_{1,2}(t,x)>\delta)\le C \delta^{-1}(1+[u-V]_{\C^{\frac58}L_2([0,1])})\|f-b\|_{\C^{\alpha'}}.
\end{equation}

Finally, to bound $I^{n,\eps}_{1,3}$ we apply \cref{l:phitd}(iii) for the function $f_{r}:=p_{t-r}(x,\cdot)$, $r\in[0,t-\eps]$. It is easy to see that $f\in\C([0,1],\S(D))$. Thus, all the assumptions of the lemma are satisfied.  Hence, by \eqref{eq:13_10_5} applied with  $t-\eps$ in place of $t$ we get for any $\delta>0$ 
\begin{equation}\label{threed}
\lim_{n\to\infty}\P(I^{n,\eps}_{1,3}(t,x)>\delta)=0.
\end{equation}
Recalling now 
  decomposition \eqref{decompivan}, we get for any $\delta>0$, $n\in\N$
\begin{equation*}
\P(I^{\eps}_{1}(t,x)>\delta)\le
\P(I^{n,\eps}_{1,1}(t,x)>\frac\delta3)+
\P(I^{n,\eps}_{1,2}(t,x)>\frac\delta3)+\P(I^{n,\eps}_{1,3}(t,x)>\frac\delta3). 
\end{equation*}
By passing to the limit in the right-hand side of the above inequality as $n\to\infty$, and using  \eqref{ioned},  \eqref{itwod}, \eqref{threed}, we finally get 
\begin{equation*}
\P(I^{\eps}_{1}(t,x)>\delta)\le C \delta^{-1}(1+[u-V]_{\C^{\frac58}L_2([0,1])})\|f-b\|_{\C^{\alpha'}},
\end{equation*}
where $C=C(\alpha,\alpha')$, which is the desired bound \eqref{desbound}.
\end{proof}

To treat  the remaining terms in \eqref{eq:decompos}, we need the following useful bound.
\begin{lemma}
\label{l:2term}
Under the assumptions of \cref{prop:10_11_1}, there exists a constant $C=C(\alpha)$ such that 
for any $\delta >0$, $(s,t)\in\Delta_{[0,1]}$, $x\in D$, $f\in\C^\infty_b$ 
\begin{equation*}
\P\Bigl(\Bigl|\int_s^{t}\int_D p_{t-r}(x,y) f(u_r(y))\,dy\,dr\Bigr|>\delta)\le C \delta^{-1}\|f\|_{\C^\alpha} (t-s)^{\frac12}(1+[u-V]_{\C^{\frac58}L_2([0,1])})
\end{equation*}
\end{lemma}
\begin{proof}
We apply \cref{L:mainml} with $\psi:=u-V$, $T:=t$, $\gamma=\alpha$, $\tau=\frac58$, $X_r:=p_{t-r}(x,\cdot)$, where $r\in[0,t)$, $\nu=0$. It is easy to see that  $\alpha+4\tau>1$ because $\alpha>-3/2$ and 
\begin{equation*}
\int_D p_{t-r}(x,y)\,dy=1.
\end{equation*}
Thus,  all the assumptions of \cref{L:mainml} are satisfied and we get 
\begin{equation*}
\Bigl\|\int_s^{t}\int_D p_{t-r}(x,y) f(u_r(y))\,dy\,dr\Bigr\|_{L_2(\Omega)}\le C \|f\|_{\C^\alpha} (t-s)^{\frac12}(1+[u-V]_{\C^{\frac58}L_2([0,1])}),\quad x\in D,
\end{equation*}
for $C=C(\alpha)$ independent of $s,t,x,f$. The statement of the lemma follows now from the Chebyshev inequality.
\end{proof}

Now we are ready to prove \cref{prop:10_11_1}.
\begin{proof}[Proof of \cref{prop:10_11_1}]
Fix $\delta>0$. Let $n\in\N$, $t\in(0,1]$, $x\in D$, $\alpha'\in(-\frac32,\alpha)$. We apply 
\cref{lem:3_11_1} and \cref{l:2term}
 and use decomposition \eqref{eq:decompos} with $b^n$ in place of $f$. We get 
for any $\eps\in(0,t]$
\begin{align}\label{eq:newdec}
&\P\Bigl(\Bigl|\int_0^t\int_D p_{t-r}(x,y) b^n(u_r(y))\,dy\,dr-R_\K(t,x)\Bigr|>\delta\Bigr)\nn\\
&\qquad\le
\P(I_1^{\eps}(t,x)>\frac\delta3)+\P(I_2^{\eps}(t,x)>\frac\delta3)+\P(I_3^{\eps}(t,x)>\frac\delta3)\nn\\
&\qquad\le C \delta^{-1}(1+[u-V]_{\C^{\frac58}L_2([0,1])})\bigl(\|b^n-b\|_{\C^{\alpha'}}+\|b^n\|_{\C^\alpha}\eps^{\frac12}\bigr)+\P(I_3^{\eps}(t,x)>\frac\delta3)
\end{align}
for $C=C(\alpha,\alpha')$. By \cref{l:txrepr}, $I_3^{\eps}(t,x)\to 0$ in probability as $\eps\to0$. 
Therefore, we can pass to the limit in the right-hand side of \eqref{eq:newdec} as $\eps\to0$ and derive for any $t\in(0,1]$, $x\in D$
\begin{equation*}
\P\Bigl(\Bigl|\int_0^t\int_D p_{t-r}(x,y) b^n(u_r(y))\,dy\,dr-R_\K(t,x)\Bigr|>\delta\Bigr)\le C \delta^{-1}(1+[u-V]_{\C^{\frac58}L_2([0,1])})\|b^n-b\|_{\C^{\alpha'}}
\end{equation*}
for $C=C(\alpha,\alpha')$ independent of $n$.
Since $b^n\to b$ in $\C^{\alpha-}$, we have $\|b^n-b\|_{\C^{\alpha'}}\to0$ as $n\to\infty$. Since $u\in\V(\frac58)$, we have $[u-V]_{\C^{\frac58}L_2([0,1])}<\infty$. Thus, \eqref{eq:3_11_1a} holds.
\end{proof}

\begin{proof}[Proof of \cref{thm:main0}: (a) implies (b)]
Let $u\in\V(\frac58)$ be a  weak regularized solution  to equation~\eqref{SPDE}. We apply \cref{l:43} with $\gamma=\alpha$ and $\tau=5/8$ and $f=b$. We see that condition \eqref{gammatau} holds since $\alpha>-\frac32$. Hence there exists a  jointly continuous adapted process $H^b\colon[0,1]\times D\times \Omega\to\R$ such that for any sequence of functions $(b_n)_{n\in\Z_+}$ converging to $b$ in $\C^{\alpha-}$ we have
\begin{equation}\label{rez210}
	\sup_{t\in[0,1]}\sup_\DN
	\lt|\int_0^t\int_\0 p_{t-r}(x,y) b^n(u_r(y))\,dy\,dr-H^b_t(x)\rt|\to0\quad  \text{in probability as $n\to\infty$.}
\end{equation}
On the other hand, by \cref{l:txrepr} $u$ satisfies $\P$-a.s.
\begin{equation}\label{uequ}
	u_t(x)= P_tu_{0}(x)   +R_\K(t,x)+V_t(x),\quad  t\in[0,1],\,\, x\in D. 
\end{equation}
Further, by \cref{prop:10_11_1} we have for any fixed $t\in[0,1]$, $x\in D$
\begin{equation*}
	\int_0^t\int_D p_{t-r}(x,y) b^n(u_r(y))\,dydr\to R_\K(t,x),\quad  \text{in probability as $n\to\infty$.}
\end{equation*}
Combining this with \eqref{rez210}, we see that there exists a set $\Omega'\subset \Omega$ of full probability measure such that $H_t^b(x)=R_{\K}(t,x)$ for $t\in[0,1]\cap \Q$, $x\in D\cap \Q$. Since both $H^b$ and $R_\K$ are continuous we see that on  $\Omega'$ we have  $H_t^b(x)=R_{\K}(t,x)$ for all  $t\in[0,1]$, $x\in D$. Thus, 
\begin{equation*}
	\sup_{t\in[0,1]}\sup_{x\in D} 
	\lt|\int_0^t\int_\0 p_{t-r}(x,y) b^n(u_r(y))\,dy\,dr-R_\K(t,x)\rt|\to0\quad  \text{in probability as $n\to\infty$.}
\end{equation*}
Recalling now \eqref{uequ}, we see that all the conditions of \cref{Def:mildrsol} are satisfied and we can conclude that $u$ is  a mild regularized solution to equation~\eqref{SPDE}.
\end{proof}

\subsection{Proof of \cref{thm:main0}: (b) implies (a)}

The first step is to show that if $u$ is a mild regularized solution  of equation  \eqref{SPDE}, then it satisfies \eqref{idweakreg} for some drift $R_\phi$ in place of $\K_t(\phi)$. 

\begin{lemma}\label{l:weakd}
Let $\alpha>-3/2$, $b\in\C^\alpha$, 
$u_0\in\Bsp(\0)$. Let $u$ be a mild regularized solution  of equation  \eqref{SPDE} with the initial condition $u_0$ and suppose that $u\in\V(5/8)$. 

Then for any $\phi\in \S(D)$, $t\in[0,1]$, there 
exists a continuous process $R_\phi\colon[0,1]\times\Omega\to\R$ such that
\begin{equation}\label{rdef}
\sum_{i=1}^{\lfloor nt\rfloor} \langle K_{\frac{i}n}-P_{\frac1n}K_{\frac{i-1}n}, \phi\rangle\to R_\phi(t)
\end{equation}
in probability as $n\to\infty$. 
Further, $u$ satisfies a.s.
\begin{equation}\label{newdec}
\langle u_t, \phi \rangle= \langle u_{0},\phi\rangle + \int_{0}^t \langle u_s, \frac{1}{2}\Delta \phi\rangle\,ds  +R_\phi(t)+\int_{0}^t\int_D \phi(x)W(ds,dx),\quad t\in[0,1].
\end{equation}
\end{lemma}

\begin{proof} 
Fix $\phi\in\S(D)$.
First, we note that by \cref{c:globalspaceb} we have
\begin{equation*}
\sup_{t\in[0,1]}\sup_{x\in D}\frac{|u_t(x)|}{1+|x|}<\infty\quad \text{a.s.}	
\end{equation*}
and therefore all the integrals in \eqref{newdec} are well-defined as Lebesgue integrals and the process $t\mapsto\la u_t,\phi\ra$ is a.s. continuous. Hence the process
\begin{equation}\label{rphidef}
R_\phi(t):=\langle u_t, \phi \rangle-\langle u_{0},\phi\rangle - \int_{0}^t \langle u_s, \frac{1}{2}\Delta \phi\rangle\,ds  -\int_{0}^t\int_D \phi(x)W(ds,dx),\quad t\in[0,1]
\end{equation}
is also continuous and \eqref{newdec} holds. 

Now let us show \eqref{rdef}. Fix $t\in[0,1]$. 
Using identity in part (i) of \cref{Def:mildrsol}, we clearly have for any $n\in \N$, $k=1,..., n$
\begin{align*}
\langle u_{\frac{k}n}, \phi\rangle-
\langle u_{\frac{k-1}n}, P_{\frac1n}\phi\rangle&=\langle P_{\frac{k}n} u_0, \phi\rangle+ \langle K_{\frac{k}n}, \phi\rangle-\langle V_{\frac{k}n}, \phi\rangle-\langle P_{\frac{k-1}n} u_0, P_{\frac1n}\phi\rangle- \langle K_{\frac{k-1}n}, P_{\frac1n}\phi\rangle\\
&\phantom{=}+\langle V_{\frac{k-1}n}, P_{\frac1n}\phi\rangle\\
&= \langle K_{\frac{k}n}-P_{\frac1n}K_{\frac{k-1}n}, \phi\rangle+\langle V_{\frac{k}n}-P_{\frac1n}V_{\frac{k-1}n}, \phi\rangle.
\end{align*}
Summing these identities over $k$, we derive for any $t\in[0,1]$
\begin{align}\label{identw}
\langle u_{\kappa_n(t)}, \phi\rangle=& \langle u_{0},\phi\rangle + \sum_{i=1}^{\lfloor nt\rfloor} 
		\langle u_{\frac{i-1}{n}},P_{\frac1n} \phi-\phi\rangle
	+	\sum_{i=1}^{\lfloor nt\rfloor} \langle K_{\frac{i}n}-P_{\frac1n}K_{\frac{i-1}n}, \phi\rangle\nn\\
	&+	\sum_{i=1}^{\lfloor nt\rfloor} \langle V_{\frac{i}n}-P_{\frac1n}V_{\frac{i-1}n}, \phi\rangle.
\end{align}
Let us pass to the limit as $n\to\infty$ in this identity. Fix $N>0$. 
We have 
\begin{equation}\label{onelimw}
\bigl|\langle u_{\kappa_n(t)}-u_t, \phi\rangle\bigr|\le 
\langle |u_{\kappa_n(t)}|+|u_t|, |\phi|\I_{|\cdot|\ge N}\rangle+\bigl|\langle u_{\kappa_n(t)}, \phi\I_{|\cdot|\le N}\rangle- \langle u_t, \phi\I_{|\cdot|\le N}\rangle\bigr|.
\end{equation}
 Using that  $u$ is a.s. bounded and continuous  on $D\cap[-N,N]$, we easily get
\begin{equation}\label{onelimwN}
\bigl|\langle u_{\kappa_n(t)}, \phi\I_{|\cdot|\le N}\rangle- \langle u_t, \phi\I_{|\cdot|\le N}\rangle\bigr|\to0,
\end{equation}
a.s. as $n\to\infty$. Since 
\begin{equation*}
\E \langle |u_{\kappa_n(t)}|+|u_t|, |\phi|\I_{|\cdot|\ge N}\rangle\le C\|u\|_{\Ctimespace{0}{1}{[0,1]}}\int_{|x|\ge N} |\phi(x)|\,dx\to0,
\end{equation*}
as $N\to\infty$, we derive from \eqref{onelimw} and \eqref{onelimwN}
\begin{equation}\label{onelimwf}
\langle u_{\kappa_n(t)}, \phi\rangle\to \langle u_{t}, \phi\rangle,\quad  \text{in probability as $n\to\infty$.}
\end{equation}

Next, we note 
\begin{equation}\label{indentmw}
 \sum_{i=1}^{\lfloor nt\rfloor} 
		\langle u_{\frac{i-1}{n}},P_{\frac1n}\phi-\phi\rangle=\int_0^{\kappa_n(t)} \langle u_{\kappa_n(s)}, \frac12\Delta P_{s-\kappa_n(s)}\phi\rangle\,ds=
		\int_0^{\kappa_n(t)} \langle u_{\kappa_n(s)},  \frac12P_{s-\kappa_n(s)}\Delta\phi\rangle\,ds.
\end{equation}
We split the integration in the scalar product as before into two parts: over the compact set $D \cap [-N, N]$ and over the remainder $D \cap (\R \setminus [-N, N])$. Then  for any fixed $N>0$ we have a.s.
\begin{equation}\label{twolimw1}
\lim_{n\to\infty }\Bigl|
\int_0^{\kappa_n(t)} \langle u_{\kappa_n(s)}\I_{|\cdot|\le N},  \frac12P_{s-\kappa_n(s)}\Delta\phi\rangle\,ds-\int_0^t \langle u_s\I_{|\cdot|\le N}, \frac12\Delta \phi\rangle\,ds
\Bigr|=0
\end{equation}
On the other hand, we get 
\begin{align*}
&\E 	\Bigl|\int_0^{\kappa_n(t)} \langle u_{\kappa_n(s)}\I_{\|\cdot\|\ge N},  \frac12P_{s-\kappa_n(s)}\Delta\phi\rangle\,ds\Bigr|\le \|u\|_{\Ctimespace{0}{1}{[0,1]}} \int_{|x|\ge N} \sup_{r\in[0,1]} P_r|\Delta \phi|(x)\,dx\to0;\\
&\E 	\Bigl|\int_0^{\kappa_n(t)} \langle u_{s}\I_{\|\cdot\|\ge N},  \frac12\Delta\phi\rangle\,ds\Bigr|\le \|u\|_{\Ctimespace{0}{1}{[0,1]}} \int_{|x|\ge N} |\Delta \phi(x)|\,dx\to0
\end{align*}
as $N\to\infty$. Combining this with \eqref{indentmw} and \eqref{twolimw1}, we get
\begin{equation}\label{twolimwf}
 \sum_{i=1}^{\lfloor nt\rfloor} 
\langle u_{\frac{i-1}{n}},P_{\frac1n}\phi-\phi\rangle\to 	\int_0^t \langle u_s, \frac12\Delta \phi\rangle\,ds,\quad  \text{in probability as $n\to\infty$.}
\end{equation}
Finally, we derive
\begin{align*}
\sum_{i=1}^{\lfloor nt\rfloor} \langle V_{\frac{i}n}-P_{\frac1n}V_{\frac{i-1}n}, \phi\rangle&=
\sum_{i=1}^{\lfloor nt\rfloor} 
\int_D\int_{t_{i-1}}^{t_i} \int_D p_{t_i-r}(x,y) W(dr,dy)\phi(x)dx\\
&=\int_{0}^{\kappa_n(t)}\int_D P_{\kappa_n(r)+\frac1n-r} \phi(y)W(dr,dy).
\end{align*}
It is not difficult to show that for any $r\in[0,t]$
\begin{equation*}
\lim_{n\to\infty}\int_D  (P_{\kappa_n(r)+\frac1n-r} \phi(y)-\phi(y))^2\,dy=0.
\end{equation*}
Thus, using Ito's isometry, we get  
\begin{equation*}
\sum_{i=1}^{\lfloor nt\rfloor} \langle V_{\frac{i}n}-P_{\frac1n}V_{\frac{i-1}n}, \phi\rangle\to \int_{0}^t\int_D \phi(y)W(dr,dy),\quad  \text{in $L_2(\Omega)$ as $n\to\infty$.}
\end{equation*}
Recalling now  \eqref{onelimwf}, \eqref{twolimwf} and substituting them into \eqref{identw}, we 
get 
\begin{align*}
\sum_{i=1}^{\lfloor nt\rfloor} \langle K_{\frac{i}n}-P_{\frac1n}K_{\frac{i-1}n}, \phi\rangle&=	\langle u_{\kappa_n(t)}, \phi\rangle- \langle u_{0},\phi\rangle - \sum_{i=1}^{\lfloor nt\rfloor} 
	\langle u_{\frac{i-1}{n}},P_{\frac1n} \phi-\phi\rangle
-	\sum_{i=1}^{\lfloor nt\rfloor} \langle V_{\frac{i}n}-P_{\frac1n}V_{\frac{i-1}n}, \phi\rangle\\
	&\to \langle u_{t}, \phi\rangle- \langle u_{0},\phi\rangle-\int_0^t \langle u_s, \frac12\Delta \phi\rangle\,ds-\int_{0}^t\int_D \phi(y)W(dr,dy)
\end{align*}
in probability as $n\to\infty$. Comparing this with the definition of $R_\phi(t)$ in \eqref{rphidef}, we get \eqref{rdef}.
\end{proof}

To complete the proof that a mild regularized solution is a weak regularized solution, we need to show that the process $R_\phi$ satisfies condition \eqref{Kcurlydef}, that is, it is a limit of the corresponding approximating sequence. We already know from \cref{l:43}(ii) that the approximating sequence converges, it remains just to identify the limit.

\begin{lemma}\label{l:weakid}
Let $\alpha>-3/2$, $b\in\C^\alpha$, 
$u_0\in\Bsp(\0)$. Let $u$ be a mild regularized solution  of equation  \eqref{SPDE} with the initial condition $u_0$ and suppose that $u\in\V(5/8)$.

Then for any  sequence of functions $(b^m)_{m\in\Z_+}$ in
      $\C_b^\infty$ such that $b^m\to b$ in $\C^{\alpha-}$, $\phi\in\S(D)$, $t\in[0,1]$ we have
\begin{equation}
\label{wlim}
\int_0^t\int_D \phi(x) b^m(u_r(x))\,dxdr\to R_\phi(t)\quad  \text{in probability as $n\to\infty$.}
\end{equation}
\end{lemma}

\begin{proof}
Fix $t\in[0,1]$, a sequence of functions $(b^m)_{m\in\Z_+}$ in
      $\C_b^\infty$ such that $b^m\to b$ in $\C^{\alpha-}$, $\phi\in\S(D)$, and   $\delta >0$.
  We write for  $n\in\N$
  \begin{align}\label{decweak}
 &\Bigl|\int_0^t\int_D \phi(y) b^m(u_r(y))\,dydr-R_\phi(t)\Bigr|\nn\\
 &\quad\le
 \Bigl|\int_0^{\kappa_n(t)}\int_D \phi(y) b^m(u_r(y))\,dxdr-\int_0^{\kappa_n(t)}\int_D\int_D \phi(x) p_{\kappa_n(r)+\frac1n-r}(x,y)b^m(u_r(y))\,dxdydr\Bigr|\nn\\
 &\quad\phantom{\le}+\Bigl|\int_0^{\kappa_n(t)}\int_D\int_D \phi(x) p_{\kappa_n(r)+\frac1n-r}(x,y)b^m(u_r(y))\,dxdydr-\sum_{i=1}^{\lfloor nt\rfloor} \langle K_{\frac{i}n}-P_{\frac1n}K_{\frac{i-1}n}, \phi\rangle\Bigr|\nn\\
 &\quad\phantom{\le}+\Bigl|\sum_{i=1}^{\lfloor nt\rfloor} \langle K_{\frac{i}n}-P_{\frac1n}K_{\frac{i-1}n}, \phi\rangle-R_\phi(t)\Bigr|+ \Bigl|\int_{\kappa_n(t)}^t\int_D \phi(y) b^m(u_r(y))\,dxdr\Bigr|\nn\\
 &=:I^{n,m}_1(t)+I_2^{n,m}(t)+I_3^n(t)+I_4^{n,m}(t).
 \end{align}

To bound $I^{n,m}_1$ we use again \cref{L:mainml} with $f=b^m$, $\gamma=\alpha$, $\tau=\frac58$, $T=\kappa_n(t)$, $X_r(y):= \phi(y)-\int_D \phi(x) p_{\kappa_n(r)+\frac1n-r}(x,y)\,dx$, $\nu=0$. We get for any $r\in[0,t]$
\begin{align*}
\int_D \Bigl|\phi(y)-\int_D \phi(x) p_{\kappa_n(r)+\frac1n-r}(x,y)\,dx\Bigr|\,dy&\le \sup_{s\in[0,\frac1n]}
\int_D |\phi(y)-P_s\phi(y)|\,dy\\
&\le Cn^{-1}\|\Delta \phi\|_{L_1(D)},
\end{align*}
for universal $C>0$,
where we also used that $p_r(x,y)=p_r(y,x)$ for any $r>0$, $x,y\in D$. Therefore, condition \eqref{con.k} holds with $\Gamma=Cn^{-1}\|\Delta \phi\|_{L_1(D)}$. Thus, all the assumptions of \cref{L:mainml} are satisfied and we deduce from \eqref{Bound1K}
\begin{equation}\label{ieins}
\|I^{n,m}_1(t)\|_{L_2(\Omega)}\le C \|b^m\|_{\C^\alpha} n^{-1}\|\Delta \phi\|_{L_1(D)}(1+[u-V]_{\C^{\frac58}L_2([0,1])}).
\end{equation}

Next, to bound $I_2^{n,m}(t)$, we rewrite it as 
\begin{align}\label{itwoww}
|I_2^{n,m}(t)|&\le \sum_{i=1}^{\lfloor nt\rfloor} \Bigl\langle \bigl|K_{\frac{i}n}-\int_0^{\frac{i}n}\int_D p_{\frac{i}n-r}(\cdot,y)b^m(u_r(y))\,dydr\bigr|, |\phi|\Bigr\rangle\nn\\
&\phantom{\le}+\sum_{i=1}^{\lfloor nt\rfloor} \Bigl\langle  \int_D P_{\frac1n}(\cdot,z)\Bigl|K_{\frac{i-1}n}(z)-\int_0^{\frac{i-1}n}\int_D p_{\frac{i-1}n-r}(z,y)b^m(u_r(y))\,dydr\Bigr|dz, |\phi|\Bigr\rangle.
\end{align}
By  part (2) of \cref{Def:mildrsol}, we have for each $i=1,...,n$
\begin{equation}\label{limpr}
K_{\frac{i}n}-\int_0^{\frac{i}n}\int_D p_{\frac{i}n-r}(\cdot,y)b^m(u_r(y))\,dydr=\lim_{M\to\infty}
\int_0^{\frac{i}n}\int_D p_{\frac{i}n-r}(\cdot,y)(b^M-b^m)(u_r(y))\,dydr,
\end{equation}
where the limit is in probability. Hence, by the Fatou's lemma and Minkowski's inequality,
\begin{align}\label{part1dif}
&\Bigl\| \Bigl\langle \bigl|K_{\frac{i}n}-\int_0^{\frac{i}n}\int_D p_{\frac{i}n-r}(\cdot,y)b^m(u_r(y))\,dydr\bigr|, |\phi|\Bigr\rangle\Bigr\|_{L_2(\Omega)}\nn\\
&\qquad\le \lim_{M\to\infty}
\Bigl\| \Bigl\langle \bigl|\int_0^{\frac{i}n}\int_D p_{\frac{i}n-r}(\cdot,y)(b^M-b^m)(u_r(y))\,dydr\bigr|, |\phi|\Bigr\rangle\Bigr\|_{L_2(\Omega)}\nn\\
&\qquad\le \lim_{M\to\infty}
 \Bigl\langle \Bigl\|\int_0^{\frac{i}n}\int_D p_{\frac{i}n-r}(\cdot,y)(b^M-b^m)(u_r(y))\,dydr\Bigr\|_{L_2(\Omega)}, |\phi|\Bigr\rangle.
\end{align}
Similarly, we use \eqref{limpr} to  treat the the second group of terms in the right-hand side of \eqref{itwoww}. We get for each $i=1,...,n$
\begin{align}\label{part2dif}
&\Bigl\| \Bigl\langle  \int_D P_{\frac1n}(\cdot,z)\Bigl|K_{\frac{i-1}n}(z)-\int_0^{\frac{i-1}n}\int_D p_{\frac{i-1}n-r}(z,y)b^m(u_r(y))\,dydr\Bigr|dz, |\phi|\Bigr\rangle\Bigr\|_{L_2(\Omega)}\nn\\
&\qquad\le \lim_{M\to\infty}\Bigl\| \Bigl\langle  \int_D P_{\frac1n}(\cdot,z)\Bigl|\int_0^{\frac{i-1}n}\int_D p_{\frac{i-1}n-r}(z,y)(b^M-b^m)(u_r(y))\,dydr\Bigr|dz, |\phi|\Bigr\rangle\Bigr\|_{L_2(\Omega)}\nn\\
&\qquad\le \lim_{M\to\infty}
 \Bigl\langle  \int_D P_{\frac1n}(\cdot,z)\Bigl\|\int_0^{\frac{i-1}n}\int_D p_{\frac{i-1}n-r}(z,y)(b^M-b^m)(u_r(y))\,dydr\Bigr\|_{L_2(\Omega)}dz, |\phi|\Bigr\rangle.
\end{align}
We fix now $\alpha'\in(-\frac32,\alpha)$ and apply \cref{L:mainml} with $f=b^M-b^m$, $\gamma=\alpha'$, $\tau=\frac58$, $T=\frac{i}n$, $X_r(y):= p_{\frac{i}n-r}(x,y)$, $\nu=0$. Since the function $b^M-b^m$ is bounded,  $\int_D p_{\frac{i}n-r}(x,y)\,dy=1$ and $\alpha'+4\tau>1$, we see that all the conditions of the proposition are satisfied and we get for each $i=1,...,n$, $z\in D$
\begin{equation*}
\Bigl\|\int_0^{\frac{i}n}\int_D p_{\frac{i}n-r}(z,y)(b^M-b^m)(u_r(y))\,dydr\Bigr\|_{L_2(\Omega)}\le 
C\|b^M-b^m\|_{\C^{\alpha'}}(1+[u-V]_{\C^{\frac58}L_2([0,1])}),
\end{equation*}
for $C=C(\alpha,\alpha')$. 
We substitute this into \eqref{part1dif}, \eqref{part2dif} and then sum over $i$ to combine with \eqref{itwoww}. Using that $b^M\to b$ in $\C^{\alpha-}$ and thus $b^M\to b$ in $\C^{\alpha'}$, we deduce 
\begin{equation}\label{izwei}
\|I_2^{n,m}(t)\|_{L_2(\Omega)}\le Cn \|b-b^m\|_{\C^{\alpha'}}\|\phi\|_{L_1(D)}(1+[u-V]_{\C^{\frac58}L_2([0,1])}),
\end{equation}
for $C=C(\alpha,\alpha')$.

Next, we use \cref{l:weakd} and \eqref{rdef} to bound $I_3^{n}(t)$. We get for any $\delta>0$
\begin{equation}\label{idrei}
\lim_{n\to\infty} \P(|I_3^n(t)|>\delta)=0.
\end{equation}

Finally, we treat $I_4^{n,m}$ using again  \cref{L:mainml} with $f=b^m$, $\gamma=\alpha$, $\tau=\frac58$, $T=t$, $X_r(y):= \phi$, $\nu=0$. Since $t-\kappa_n(t)\le n^{-1}$, we get 
\begin{equation*}
\|I^{n,m}_4(t)\|_{L_2(\Omega)}\le C \|b^m\|_{\C^\alpha} n^{-\frac12}\|\phi\|_{L_1(D)}(1+[u-V]_{\C^{\frac58}L_2([0,1])})
\end{equation*}
for $C=C(\alpha)$. Now we substitute this together with \eqref{ieins}, \eqref{izwei}, \eqref{idrei} into the decomposition \eqref{decweak}, use the Chebyshev's inequality and get for any $n,m\in\N$, $\delta>0$
\begin{align*}
&
\P\Bigl(\Bigl|\int_0^t\int_D \phi(x) b^m(u_r(x))\,dxdr-R_\phi(t)\Bigr|\ge\delta\Bigl)\\
&\quad\le 
\P(I_1^{n,m}(t)>\frac\delta4)+
\P(I_2^{n,m}(t)>\frac\delta4)+\P(I_3^{n}(t)>\frac\delta4)+\P(I_4^{n,m}(t)>\frac\delta4)\\
&\quad \le C \delta^{-1}(\| \phi\|_{L_1(D)}+\|\Delta \phi\|_{L_1(D)})(1+[u-V]_{\C^{\frac58}L_2([0,1])})\bigl(\|b^m\|_{\C^\alpha} n^{-\frac12}+
n\|b-b^m\|_{\C^{\alpha'}}\bigr)\\
&\quad \phantom{\le}+\P(I_3^{n}(t)>\frac\delta4),
\end{align*}
for $C=C(\alpha,\alpha')$.
By passing to the limit in the above inequality as $m\to\infty$, we get 
\begin{align*}
&
\P\Bigl(\Bigl|\int_0^t\int_D \phi(x) b^m(u_r(x))\,dxdr-R_\phi(t)\Bigr|\ge\delta\Bigl)\\
&\quad\le C n^{-\frac12} \delta^{-1}
(\| \phi\|_{L_1(D)}+\|\Delta \phi\|_{L_1(D)})
(1+[u-V]_{\C^{\frac58}L_2([0,1])})\sup_m\|b^m\|_{\C^\alpha}\\ 
&\qquad+\sup_{t\in[0,1]}\P(I_3^{n}(t)>\frac\delta4),
\end{align*}
for $C=C(\alpha,\alpha')$.
Note that $u\in\V(\frac58)$ and therefore $[u-V]_{\C^{\frac58}L_2([0,1])}<\infty$. Further, $\sup_m\|b^m\|_{\C^\alpha}<\infty$ by the definition of convergence in $\C^{\alpha-}$. By passing to the limit in the above inequality as $n\to\infty$ and using \eqref{idrei}, we get the desired result \eqref{wlim}.
\end{proof}

\begin{proof}[Proof of \cref{thm:main0}: (b) implies (a)]
Let $u\in\V(\frac58)$ be a  mild regularized solution  to equation~\eqref{SPDE}. Let $\phi\in\S(D)$. Then by \cref{l:weakd} $u$ satisfies
a.s.
\begin{equation}\label{210weak}
	\langle u_t, \phi \rangle= \langle u_{0},\phi\rangle + \int_{0}^t \langle u_s, \frac{1}{2}\Delta \phi\rangle\,ds  +R_\phi(t)+\int_{0}^t\int_D \phi(x)W(ds,dx),\quad t\in[0,1].
\end{equation}
Let $(b^n)_{n\in\Z+}$ be a sequence of functions converging to $b$ in $\C^{\alpha-}$.
Applying \cref{l:43}(ii) with $\gamma=\alpha$, $\tau=5/8$, $f=b$ (note  that condition \eqref{gammatau} holds because $\gamma>-\frac32$),  we get 
that there exists a continuous process $\H^b(\phi)\colon[0,1]\times \Omega\to\R$ such that 
\begin{equation}\label{suplim}
\sup_{t\in[0,1]}
\lt|\int_0^t\int_\0 \phi(x) b^n(u_r(x))\,dx\,dr-\H^b_t(\phi)\rt|\to0\quad  \text{in probability as $n\to\infty$.}
\end{equation}
On the other hand, for any fixed $t\in[0,1]$ by \cref{l:weakid}
\begin{equation*}
	\int_0^t\int_D \phi(x) b^m(u_r(x))\,dxdr\to R_\phi(t)\quad  \text{in probability as $n\to\infty$.}
\end{equation*}
Therefore there exists a set of full probability measure $\Omega'\subset \Omega$ such that $R_\phi(t)=\H^b_t(\phi)$ for $t\in[0,1]\cap \Q$. Since both $R_\phi$ and $\H^b(\phi)$ are continuous, we see that 
\begin{equation}\label{identt}
R_\phi(t)=\H^b_t(\phi)\quad\text{ for all $t\in[0,1]$ on $\Omega'$.} 
\end{equation}
Now by \eqref{210weak}, \eqref{suplim} and \eqref{identt} we see that all the conditions of \cref{Def:weakrsol} are satisfied and  $u$ is  a weak regularized solution to equation~\eqref{SPDE}.
\end{proof}

\subsection{Proofs of \cref{thm:main2} and \cref{thm:main1}}

The key step in proving \cref{thm:main2}  is to show that if $b$ is an integrable function, then any analytically weak or mild solution of equation \eqref{SPDE} belongs to the class $\V(\kappa)$ for some $\kappa\ge\frac34$. This is done using \cref{L:metaml} which in turn relies on the stochastic sewing lemma with random controls. Once we establish that a solution belongs to this class, we can show that it is also a regularized solution and thus the uniqueness theorem for regularized solutions is applicable \cite[Theorem~2.6]{ABLM24}.
 As mentioned in \cref{s:MR}, it is relatively straightforward to show that any analytically weak solution to \eqref{SPDE} satisfies \eqref{SPDEfunc} for Lebesgue a.e. $x\in D$. Therefore, the following statement is crucial.

\begin{lemma}\label{l:class}
Let $p\in[1,\infty]$, $b\in L_p(\R)$, $u_0\in\Bsp(\0)$. Let $u\colon(0,1]\times D\times\Omega\to\R$ be a jointly continuous adapted process. 
Suppose that for any $t\in[0,1]$, for Lebesgue a.e. $x\in D$ we have
\begin{equation}\label{cond1}
\int_0^t \int_D p_{t-r}(x,y)|b|(u_r(y))\,dydr<\infty\quad \text{$\P$-a.s.}
\end{equation}
and 
\begin{equation}\label{cond2}
u_t(x)= P_t u_0(x)+\int_0^t \int_D p_{t-r}(x,y)b(u_r(y))\,dydr+V_t(x)\quad \text{$\P$-a.s.}
\end{equation}
Then $u\in \V(1-\frac1{4p})$.
\end{lemma}

To prove this lemma introduce the following set:
\begin{equation}\label{setdt}
	D_t:= \{x\in D: \text{assumptions \eqref{cond1} and \eqref{cond2} hold}\},
\end{equation}
where $t\in[0,1]$.
We show \cref{l:class} in two steps: first we show that inequality \eqref{kdefineq} in the definition of class $\V$ holds for Lebesgue a.e. $x\in D$, and then we show this inequality for all $x\in D$ using continuity of $u$ and Fatou's lemma.

\begin{lemma}\label{l:49}
Assume that the assumptions of \cref{l:class} are satisfied. Then for any $(s,T)\in\Delta_{[0,1]}$, $x\in D_T$ we have
\begin{equation*}
\Bigl\|\int_s^T \int_D p_{T-r}(x,y)|b|(u_r(y))\,dydr\Bigr\|_{L_m(\Omega)}\le C|T-s|^{1-\frac1{4p}}.
\end{equation*}
for $C=C(m,p,\|b\|_{L_p(\R)})>0$ independent of $x$.
\end{lemma}
\begin{proof}

Fix $T\in[0,1]$. First, consider the case $p<\infty$. 
We apply \cref{L:metaml} for $\gamma=-\frac1p$, $f=|b|$, $\psi:=u-V$ and put  for $(s,t)\in\Delta_{[0,T]}$, $x\in D$
\begin{equation}\label{wstform}
w_{s,t}(x):=\I_{x\in D_t}\int_s^t\int_\0 p_{t-r}(x,y) |b|(u_r(y))\,dydr.
\end{equation}
Let us verify conditions \eqref{wdef1}--\eqref{wdef2}. Let $t\in[0,T]$. Then it follows from the definition of the set $D_t$ that for $x\in D_t$ we have
\begin{align*}
|\psi_t(x)-P_{t-s}\psi_s(x)|&=\Bigl|\int_s^t\int_\0 p_{t-r}(x,y) b(u_r(y))\,dydr\Bigr|\\
&\le\int_s^t\int_\0 p_{t-r}(x,y) |b|(u_r(y))\,dydr= w_{s,t}(x)
\end{align*}
and thus \eqref{wdef1} holds. Further for any $(s,u)\in \Delta_{[0,t]}$, $x\in D_t$ we have 
\begin{align*}
P_{t-u}w_{s,u}(x)+w_{u,t}(x)&=
\int_s^u\int_\0 p_{t-r}(x,y) |b|(u_r(y))\,dydr+\int_u^t\int_\0 p_{t-r}(x,y) |b|(u_r(y))\,dydr\\
&=w_{s,t}(x).
\end{align*}
Hence condition \eqref{wdef2} holds. We note also that the function $|b|$ is nonnegative. Thus, all the conditions of  \cref{L:metaml} are satisfied and therefore using the embedding $L_p(\R)\subset \C^{-\frac1p}$  we have
for any $(s,t)\in\Delta_{[0,T]}$, $x\in D_T$
\begin{align}\label{idwst}
P_{T-t}w_{s,t}(x)&=\int_s^t\int_\0 p_{T-r}(x,y) |b|(V_r(y)+\psi_r(y))\,dydr\nn\\
&\le 
  C_0\|b\|_{L_p(\R)}(t-s)^{\frac34-\frac1{4p}}P_{T-t}w_{s,t}(x)+L_{s,t}(x),
\end{align}
where $C_0=C_0(p)$, and for any $m\ge 2$
\begin{equation*}
  \|L_{s,t}(x)\|_{\Lm}\le C_1\|b\|_{L_p(\R)}|t-s|^{1-\frac1{4p}},
\end{equation*}
for $C_1=C_1(m,p)$.
Let $\ell>0$ be such that 
\begin{equation*}
  C_0\|b\|_{L_p(\R)}\ell^{\frac34-\frac1{4p}}	 \le \frac12.
\end{equation*}
Then \eqref{idwst} yields for $(s,t)\in\Delta_{[0,T]}$ with $t-s\le \ell$
\begin{equation*}
  \|P_{T-t}w_{s,t}(x)\|_{\Lm}\le 2C_1\|b\|_{L_p(\R)}|t-s|^{1-\frac1{4p}}.
\end{equation*}
By splitting the interval $[s,T]$, where $s\in[0,T]$, into $\lceil \frac{T-s}\ell\rceil$ disjoint intervals of length no bigger than $\ell$ and applying the above inequality to each of the intervals we get for $x\in D_T$
\begin{equation*}
\Bigl\| \int_s^T\int_\0 p_{T-r}(x,y) |b|(V_r(y)+\psi_r(y))\,dydr \Bigr\|_{L_m(\Omega)}	\le C (T-s)^{1-\frac1{4p}},
\end{equation*}
for $C=C(m,p,\|b\|_{L_p(\R)})$, which is the desired bound.

Now we treat the case $p=\infty$. Take $M:=\|b\|_{L_\infty(\R)}$. Then ${b_M(x):=|b(x)|\I(|b(x)|\!>\!M)}$ equals $0$ for Lebesgue a.e. $x\in\R$ and $\|b_M\|_{\C^{-1}}\le \|b_M\|_{L_1(\R)}=0$. Therefore, \cref{L:metaml} with $\gamma=-1$, $f=b_M$, $\psi=u-V$ and $w_{s,t}$ as in \eqref{wstform} implies
for any $(s,T)\in\Delta_{[0,1]}$, $x\in D_T$
\begin{equation*}
\Bigl\|\int_s^T \int_D p_{T-r}(x,y)b_M(u_r(y))\,dydr\Bigr\|_{L_m(\Omega)}=0.
\end{equation*}
Hence for any $(s,T)\in\Delta_{[0,1]}$, $x\in D_T$ we finally get 
\begin{align*}
	\Bigl\|\int_s^T \int_D p_{T-r}(x,y)|b|(u_r(y))\,dydr\Bigr\|_{L_m(\Omega)}\le& 	\Bigl\|\int_s^T \int_D p_{T-r}(x,y)|b|(u_r(y))\I_{|b|(u_r(y))\le M}\,dydr\Bigr\|_{L_m(\Omega)}\\
	&+	\Bigl\|\int_s^T \int_D p_{T-r}(x,y)b_M(u_r(y))\,dydr\Bigr\|_{L_m(\Omega)}\\
 &\le M |T-s|.
\end{align*}
This implies the statement of the lemma.
\end{proof}

Now we are ready to prove \cref{l:class}.
\begin{proof}[Proof of \cref{l:class}]
Let the set $D_t$ be as in \eqref{setdt}. 
Fix arbitrary $(s,t)\in\Delta_{[0,1]}$, $x\in D_t$. Then for any $m\ge2$ we have by \cref{l:49}
\begin{align*}	
\|u_t(x)-V_t(x) - P_{t-s}(u_s-V_s)(x)\|_{L_m(\Omega)}&=\Bigl\|\int_s^t \int_D p_{t-r}(x,y)b(u_r(y))\,dydr\Bigr\|_{L_m(\Omega)}\\
&\le C|t-s|^{1-\frac1{4p}},
\end{align*}	
for $C=C(m,p,\|b\|_{L_p(\R)})>0$ independent of $x$. Since $u$ and $V$ are continuous in $x$, Fatou's lemma implies 
\begin{equation}\label{vclassu}	
\|u_t(x)-V_t(x) - P_{t-s}(u_s-V_s)(x)\|_{L_m(\Omega)}\le C|t-s|^{1-\frac1{4p}},
\end{equation}	
for all $x\in D$. Thus, \eqref{vclassu} holds for all  $(s,t)\in\Delta_{[0,1]}$ and $x\in D$ and therefore $u\in\V(1-\frac1{4p})$.
\end{proof}

Next, as announced above, we show that any analytically weak solution to \eqref{SPDE} satisfies \eqref{SPDEfunc} for Lebesgue a.e. $x\in D$.

\begin{lemma}	\label{l:weakf}
Let $b$ be a measurable function $\R\to\R$, 
$u_0\in\Bsp(\0)$. Let $u$ be an analytically weak  solution  of equation  \eqref{SPDE} with the initial condition 
$u_0$. Let $f\in\C([0,1],\S(D))$ and assume additionally that 
\begin{equation}\label{ass_f}
\sup_{t\in[0,1]}\sup_{x\in D}|f_t(x)|e^{|x|}<\infty.
\end{equation}
 Then  we have for any $t\in[0,1]$
\begin{align}\label{identityweak}
\langle u_t, f_t\rangle&= \langle u_{0},f_{0}\rangle + \int_{0}^t \langle u_r, \frac{\partial  f_r}{\partial r}+ \frac{1}{2}\Delta f_r\rangle\,dr  +\int_0^t\int_D b(u_r(x)) f_r(x)\,dxdr\nn\\
&\phantom{=}+\int_{0}^t\int_D f_{r}(x)W(dr,dx),\quad \P-a.s.
\end{align}
Moreover, for any $\phi\in\S(D)$ such that $\sup_{x\in D}|\phi(x)|e^{|x|}<\infty$, we have for any $t\in[0,1]$
\begin{equation}\label{phires}
\langle u_t,\phi\rangle= \langle u_{0},P_t \phi\rangle   +\int_0^t\int_D P_{t-r}\phi(x)b(u_r(x))\,dxdr
+\int_{0}^t\int_D P_{t-r}\phi(x)W(dr,dx),\,\, \P-a.s.
\end{equation}
\end{lemma}

\begin{proof}
First of all we note that a.s.
\begin{equation}\label{lintegr}
\int_0^1\int_D |b(u_r(y))|e^{-|y|}\,dydr<\infty.
\end{equation}
Indeed, for all the three choices of $(D,\S(D))$, see \cref{c:conv}, one can find a test function $\phi\in\S(D)$, such that $\phi(x)\ge e^{-|x|}$, $x\in D$. Thus, the integral in \eqref{lintegr}
is well-defined as a Lebesgue integral by definition of an analytically weak solution $u$.

We apply \cref{l:phitd}(iii). Since $u$ is an analytically weak solution to equation \eqref{SPDE},
identity  \eqref{uid} holds for $\H_t(\phi):=\int_0^t\int_D b(u_r(y))\phi(y)dydr$, where $\phi\in\S(D)$. Therefore all conditions of 
\cref{l:phitd} are satisfied and therefore \eqref{eq:13_10_5} and \eqref{identl41} holds. Let us now identify the process  $\int_0^t \H(ds,f_s)$ from \eqref{eq:13_10_5}.

For  $f\in\C([0,1],\S(D))$, $t\in[0,1]$, we have 
\begin{equation}\label{id1born}
\sum_{i=1}^{\lfloor nt\rfloor} \H_{\frac{i}n}(f_{\frac{i-1}n})-\H_{\frac{i-1}n}(f_{\frac{i-1}n})=\int_0^{\kappa_n(t)}\int_D b(u_{r}(y))f_{\kappa_n(r)}(y)dydr.
\end{equation}
We see that for any $r\in[0,t]$, $y\in D$ we have $b(u_{r}(y))f_{\kappa_n(r)}(y)\I_{r\le \kappa_n(t)}\to b(u_{r}(y))f_{r}(y)$ as $n\to\infty$ thanks to the continuity of $f$. Moreover, by assumption \eqref{ass_f} we have for a universal constant $C>0$
\begin{equation*}
|b(u_{r}(y))f_{\kappa_n(r)}(y)|\le C e^{-|y|}|b(u_{r}(y))|\in L_1([0,1]\times D)\quad \text{a.s.},
\end{equation*}
where we have also used  \eqref{lintegr}. Thus, by the dominated convergence theorem
\begin{equation}\label{id2born}
\int_0^{\kappa_n(t)}\int_D b(u_{r}(y))f_{\kappa_n(r)}(y)dydr\to
\int_0^t\int_D b(u_{r}(y))f_{r}(y)dydr,\quad \text{a.s. as $n\to\infty$}
\end{equation}
On the other hand, by  \eqref{eq:13_10_5}
\begin{equation*}
\sum_{i=1}^{\lfloor nt\rfloor} (\H_{\frac{i}n}(f_{\frac{i-1}n})-\H_{\frac{i-1}n}(f_{\frac{i-1}n}))
			\to\int_0^t \H(dr,f_r),\quad \text{in probability as $n\to\infty$}.
\end{equation*}
Recalling now \eqref{id1born}, \eqref{id2born}, we get that for any $t\in[0,1]$ we have 
\begin{equation*}
\int_0^t \H(dr,f_r)=\int_0^t\int_D b(u_{r}(y))f_r(y)dydr\quad \text{a.s.}
\end{equation*}
Therefore, \eqref{identityweak} follows now from \eqref{identl41}.
 
To show identity \eqref{phires}, we fix $t \in [0,1]$, $\phi\in\S(D)$ such that  $\sup_{x\in D}|\phi(x)|e^{|x|}<\infty$  and take $f_r := P_{t-r}\phi$, $r \in [0,t]$, in \eqref{identityweak}. This is valid because such a function $f$ belongs to $\C([0,1], \S(D))$ and satisfies \eqref{ass_f}. For this choice of $f$, we have $\frac{\partial f_r}{\partial r} + \frac{1}{2}\Delta f_r = 0$, which implies \eqref{phires}.
\end{proof}

\begin{corollary}\label{c:weak}
Under the conditions of \cref{l:weakf}, for any $t\in[0,1]$ we have for Lebesgue a.e. $x\in D$ 
\begin{equation}\label{cond11}
\int_0^t \int_D p_{t-r}(x,y)|b|(u_r(y))\,dydr<\infty\quad \text{$\P$-a.s.}
\end{equation}
and 
\begin{equation}\label{cond22}
u_t(x)= P_t u_0(x)+\int_0^t \int_D p_{t-r}(x,y)b(u_r(y))\,dydr+V_t(x)\quad \text{$\P$-a.s.}
\end{equation}
\end{corollary}
\begin{proof}
We begin by observing that  by Fubini's theorem and \eqref{lintegr} we have for any $t\in[0,1]$, $x\in D$ 
\begin{equation*}
\int_D \int_0^t \int_D e^{-|x|}p_{t-r}(x,y) |b|(u_r(y))\, dydrdx\le C  \int_0^t \int_D  e^{-|y|}|b|(u_r(y))\, dydr<\infty,
\end{equation*}
where we also applied the elementary inequality $-|x|\le |x-y|-|y|$.

That is,
\begin{equation*}
e^{-|\cdot|}\int_0^t \int_D p_{t-r}(\cdot,y) |b|(u_r(y))\, dydr\in L_1(D)\quad \P-a.s.
\end{equation*}
and thus $\P$-a.s. the function $x\mapsto \int_0^t \int_D p_{t-r}(x,y) |b|(u_r(y))\, dydr$ is finite for Lebesgue a.e. $x\in D$. This shows \eqref{cond11}.

Next, we fix $t\in[0,1]$, $x\in D$ and for an arbitrary $\eps\in(0,t)$ we take in \eqref{phires} $t-\eps$ in place of $t$ and $\phi:=p_{\eps}(x,\cdot)$. We get 
\begin{equation}\label{prelimend} 
P_{\eps}u_{t-\eps}(x)=P_t u_0(x)+\int_0^{t-\eps}\int_D p_{t-r}(x,y) b(u_r(y))\,dydr+P_{\eps}V_{t-\eps}(x),\quad \P-a.s.
\end{equation}
By \cref{l:phitd}(ii), we have $P_{\eps}u_{t-\eps}(x)\to u_t(x)$ a.s. as $\eps\to0$. Using It\^o's isometry, we have  $P_{\eps}V_{t-\eps}(x)\to V_t(x)$ in probability as $\eps\to0$. 
If $x\in D$ is such that the integral  $\int_0^t\int_D p_{t-r}(x,y)|b|(u_r(y))\,dydr$ is finite, then by definition
\begin{equation*}
\lim_{\eps\to0}\int_0^{t-\eps}\int_D p_{t-r}(x,y) b(u_r(y))\,dydr=
\int_0^{t}\int_D p_{t-r}(x,y) b(u_r(y))\,dydr.
\end{equation*}
Thus, for all such $x$ we can pass to the limit in \eqref{prelimend} as $\eps\to0$  to deduce 
\begin{equation*}
u_{t}(x)=P_t u_0(x)+\int_0^{t}\int_D p_{t-r}(x,y) b(u_r(y))\,dydr+V_{t}(x),\quad \P-a.s.,
\end{equation*}
which is \eqref{cond22}.
\end{proof}

Now, after all the preliminary work, we are ready to prove \cref{thm:main2}.

\begin{proof}[Proof of \cref{thm:main2}]

In \cref{thm:main0} we have  already  shown that (b) is equivalent to (d). Thus it remains to show that  (a) is equivalent to (b) and (c) is equivalent to (d) and that (a) implies that $u\in \V(1-\frac1{4p})$.

\textit{Proof that (a) implies $u\in \V(1-\frac1{4p})$ and (b)}. 
Let $u$ be an analytically weak solution to \eqref{SPDE}. Then
by \cref{c:weak} for any $t\in[0,1]$ we have for Lebesgue a.e. $x\in D$ 
\begin{equation*}
\int_0^t \int_D p_{t-r}(x,y)|b|(u_r(y))\,dydr<\infty\quad \text{$\P$-a.s.}
\end{equation*}
and 
\begin{equation*}
u_t(x)= P_t u_0(x)+\int_0^t \int_D p_{t-r}(x,y)b(u_r(y))\,dydr+V_t(x)\quad \text{$\P$-a.s.}
\end{equation*}
Therefore all the conditions of \cref{l:class} are satisfied and $u\in \V(1-\frac1{4p})\subset \V(\frac58)$. 

Next, note that $L_p(\R)\subset \C^{-\frac1p}\subset \C^{-1}$. Let us check that $u$ satisfies the conditions of \cref{Def:weakrsol} with $\beta=-\frac1p$, $q=\infty$. Fix any test function  $\phi \in \S(D)$. We apply \cref{l:43}(ii) with $\tau=\frac58$, $\gamma=-1$. We have that $u\in\V(\frac58)$ by above, and  condition \eqref{gammatau} obviously holds. Therefore there exists a continuous process $\H(\phi)\colon [0,1]\times \Omega\to\R$ such that for any sequence of $\C^\infty_b$ functions $(b^n)_{n\in\Z_+}$ converging to $b$ in $\C^{-1-}$ we have 
	\begin{equation}\label{tmp_h}
	\sup_{t\in[0,1]}
	\lt|\int_0^t\int_\0 \phi(y) b^n(u_r(y))\,dy\,dr-\H_t(\phi)\rt|\to0\quad  \text{in probability as $n\to\infty$.}
\end{equation}
Therefore \eqref{tmp_h} holds for  any sequence of $\C^\infty_b$ functions $(b^n)_{n\in\Z_+}$ converging to $b$ in $\C^{-\frac1p-}$.

Next, we  apply   \cref{c:mainml}  with $\gamma=-\frac5{4}$, $g=b^n$, $f=b$, $X_r(y)= \phi(y)$.  We see that $b^n$ is bounded and $b\in L_p(\R)$. Thus, all the assumptions of \cref{c:mainml} are satisfied and we deduce from \eqref{cBound1K}
\begin{equation*}
\Bigl\|\int_0^t\int_\0 \phi(y) (b-b^n)(u_r(y))\,dy\,dr\Bigr\|_{L_2(\Omega)}\le C\|\phi\|_{L_1(D)}\|b-b^n\|_{\C^{-\frac54}}(1+[u-V]_{\C^{\frac58}L_2([0,1])})
\end{equation*}
for $C=C(p)$.
By above, $u\in \V(\frac58)$ and thus $[u-V]_{\C^{\frac58}L_2([0,1])}<\infty$. By passing to the limit as $n\to\infty$, we get 
\begin{equation*}
\lim_{n\to\infty}\Bigl\|\int_0^t\int_\0 \phi(y) (b-b^n)(u_r(y))\,dy\,dr\Bigr\|_{L_2(\Omega)}=0,
\end{equation*}
where we used the embedding $L_p(\R)\subset \C^{-\frac54}$. 
Comparing this with \eqref{tmp_h},  we see that for any $t\in[0,1]$ we have 
\begin{equation}\label{htphiid}
	\H_t(\phi)=\int_0^t\int_\0 \phi(y) b(u_r(y))\,dy\,dr\quad\text{ a.s.}
\end{equation}
Since $u$ is an analytically weak solution to \eqref{SPDE}, identity \eqref{htphiid} yields that  there exists a set $\Omega'\subset \Omega$ of full probability measure such that on $\Omega'$ we have  for any $t\in[0,1]\cap\Q$
\begin{equation*}
	\langle u_t,\phi\rangle= \langle u_0,\phi\rangle +\int_0^t \langle u_s, \frac{1}{2}\Delta\phi\rangle\,ds  +\H_t(\phi)+W_t(\phi).
\end{equation*}
Since all the terms in the above identity are continuous in $t\in(0,1]$, it follows  that the above equation holds on $\Omega'$ for all $t\in[0,1]$. Therefore, recalling \eqref{tmp_h}, we see that all the conditions of \cref{Def:weakrsol} are satisfied and $u$ is a weak regularized solution of \eqref{SPDE}.

\textit{Proof that (b) implies (a)}. 
It follows from the definition of the weak regularized solution $u$ that
for any test function $\phi\in \S(\0)$  there exists a continuous process $\K(\phi)\colon[0,1]\times\Omega\to\R$ with $\K_0(\phi)=0$, such that for $t\in[0,1]$
\begin{equation}\label{decu}
\langle u_t,\phi\rangle= \langle u_0,\phi\rangle +\int_0^t \langle u_s, \frac{1}{2}\Delta\phi\rangle\,ds  +\K_t(\phi)+W_t(\phi).
\end{equation}
Set $b^n:=G_{1/n}b$. Then $b^n\to b$ in $\C^{-\frac1p-}$ as $n\to\infty$ and, by the definition of a regularized weak solution,  for any 
$t\in[0,1]$
$$
\int_0^t\int_\0 \phi(y) b^n(u_r(y))\,dy\,dr-\K_t(\phi)\to0\quad \text{in probability as $n\to\infty$}.
$$
On the other hand, applying  \cref{c:mainml}  with $\gamma=-\frac5{4}$, $g=b^n$, $f=b$, $X_r(y)= \phi(y)$,  we see as above that  for any $t\in[0,1]$
\begin{equation}\label{limlimint}
\int_0^t\int_\0 \phi(y) b^n(u_r(y))\,dy\,dr\to\int_0^t\int_\0 \phi(y) b(u_r(y))\,dy\,dr\quad \text{in $L_2(\Omega)$ as $n\to\infty$}.
\end{equation}
Therefore for every $t\in[0,1]$ we have $\K_t(\phi)=\int_0^t\int_\0 \phi(y) b(u_r(y))\,dy\,dr$ a.s.
We note that \eqref{limlimint} implies that $\int_0^1\int_\0 \phi(y) |b|(u_r(y))\,dy\,dr<\infty$ a.s. Therefore, the process $t\mapsto\int_0^t\int_\0 \phi(y) b(u_r(y))\,dy\,dr$ is a.s. continuous.
 Since $\K(\phi)$ is also  continuous, we have that a.s. for all $t\in[0,1]$ we have $\K_t(\phi)=\int_0^t\int_\0 \phi(y) b(u_r(y))\,dy\,dr$. Recalling now \eqref{decu}, we see that $u$ is an analytically weak solution to \eqref{SPDE}.

\textit{Proof that (c) implies (d)}. Let $u$ be a mild solution to \eqref{SPDE}. Then  \cref{l:class} implies $u\in\V(\frac58)$. 
Consider  any sequence of functions $(b^n)_{n\in\Z_+}$ in
$\C_b^\infty$ such that $b^n\to b$ in $\C^{-\frac1p-}$.
For a fixed $x\in D$, $t \in [0,1]$ we apply \cref{c:mainml}  with $\gamma=-\frac5{4}$, $g=b^n$, $f=b$, $X_r(y)= p_{t-r}(x,y)$.  We get 
\begin{equation*}
	\Bigl\|\int_0^t\int_\0 p_{t-r}(x,y) (b-b^n)(u_r(y))\,dy\,dr\Bigr\|_{L_2(\Omega)}\le C\|b^n-b\|_{\C^{-\frac54}}(1+[u-V]_{\C^{\frac58}L_2([0,1])}),
\end{equation*}
for $C=C(p)>0$. 
Since by above $u\in\V(\frac58)$,   we get 
$[u-V]_{\C^{\frac58}L_2([0,1])}<\infty$ and thus
\begin{equation}\label{l3outm}
\lim_{n\to\infty} \Bigl\|\int_0^t\int_\0 p_{t-r}(x,y) (b^n-b)(u_r(y))\,dy\,dr\Bigr\|_{L_2(\Omega)}=0 .
\end{equation}

On the other hand, applying \cref{l:43}(i) with $\tau=\frac58$, $\gamma=-1$, we see that there exists a continuous process $H\colon [0,1]\times D\times  \Omega\to\R$ such that for any $N>0$
\begin{equation}\label{tmp_mh}
	\sup_{t\in[0,1]}\sup_\DN
	\lt|\int_0^t\int_\0 p_{t-r}(x,y) b^n(u_r(y))\,dydr-H_t(x)\rt|\!\to0\,\,\text{in probability as $n\to\infty$.}
\end{equation}
Comparing this with \eqref{l3outm}, we see that  there exists a set $\Omega'\subset \Omega$ of full probability measure such that on $\Omega'$ we have  for any $t\in[0,1]\cap\Q$, $x\in D\cap\Q$
\begin{equation*}
	H_t(x)=\int_0^t\int_\0 p_{t-r}(x,y) b(u_r(y))dydr.
\end{equation*}
Using now \eqref{SPDEfunc} we see that on a certain set of full probability measure $\Omega''\subset \Omega'$ we have  for any $t\in[0,1]\cap\Q$, $x\in D\cap\Q$
\begin{equation*}
	u_t(x)=P_tu_0(x)+H_t(x) + V_t(x).
\end{equation*}
Since all the terms in the above identity are continuous in $(t,x)\in(0,1]\times D$, we see that it holds for all 
$t\in[0,1]$, $x\in D$. This together with \eqref{tmp_mh} shows that $u$ satisfies (1) and (2) of \cref{Def:mildrsol} with $K=H$.

\textit{Proof that  (d) implies (c).}
It follows from the definition of the solution $u$ that a.s. for any 
$x\in D$, $t\in(0,1]$
\begin{equation}\label{idmild}
u_t(x)=P_tu_0(x)+K_t(x)+V_t(x).
\end{equation}
Set $b^n:=G_{1/n}b$.  Then $b^n\to b$ in $\C^{-\frac1p-}$ as $n\to\infty$ and, by definition of a regularized mild solution,  for any $x\in\0$, $t\in(0,1]$
$$
\int_0^t\int_\0 p_{t-r}(x,y) b^n(u_r(y))\,dy\,dr-K_t(x)\to0\quad \text{in probability as $n\to\infty$}.
$$
On the other hand, applying \cref{c:mainml}  with $\gamma=-\frac5{4}$, $g=b^n$, $f=b$, $X_r(y)= p_{t-r}(x,y)$, we get 
$$
\int_0^t\int_\0 p_{t-r}(x,y) b^n(u_r(y))\,dy\,dr\to\int_0^t\int_\0 p_{t-r}(x,y) b(u_r(y))\,dy\,dr\quad \text{in $L_2(\Omega)$ as $n\to\infty$},
$$
where we also used that   $u\in\V(\frac58)$ by assumption and thus
$[u-V]_{\C^{\frac58}L_2([0,1])}<\infty$. Therefore, $K_t(x)=\int_0^t\int_\0 p_{t-r}(x,y) b(u_r(y))\,dy\,dr$ a.s. Since $u$ satisfies 
\eqref{idmild}, we get that $u$ is a mild solution to \eqref{SPDE}.
\end{proof}
\begin{proof}[Proof of \cref{thm:main1}]
By \cite[Theorem~2.6]{ABLM24}, equation \eqref{SPDE} has a mild regularized solution, which is strong in the probabilistic sense. By \cref{thm:main0} this solution is analytically mild and analytically weak.

To show strong uniqueness of solutions to equation \eqref{SPDE}, we note that any analytically mild solution as well as any analytically weak solution to this equation is a mild regularized solution by \cref{thm:main0} and belongs to the class $\V(1-\frac1{4p})\subset \V(\frac34)$. By \cite[Theorem~2.6]{ABLM24}, equation \eqref{SPDE} has a unique mild regularized solution in the class $\V(\frac34)$. Therefore, equation \eqref{SPDE} has a unique analytically mild solution, which is also analytically weak solution.	
\end{proof}

\bibliographystyle{alpha}
\bibliography{paper1}

 \end{document}